\documentclass[11pt,fleqn]{article}

\usepackage{amsmath,amssymb,amsthm,enumerate}

\newcommand{\p}{\partial}

\newcommand{\ord}{\mathop{\rm ord}\nolimits}
\newcommand{\sco}{\mathop{\rm sco}\nolimits}
\newcommand{\wsco}{\mathop{\rm wsco}\nolimits}
\newcommand{\rank}{\mathop{\rm rank}\nolimits}

\newtheorem{theorem}{Theorem}
\newtheorem{lemma}{Lemma}
\newtheorem{corollary}{Corollary}
\newtheorem{proposition}{Proposition}
{\theoremstyle{definition} \newtheorem{definition}{Definition}

\newtheorem{note}{Note}

\begin{document}

\par\noindent {\LARGE\bf
Singular Reduction Operators in Two Dimensions\par}

{\vspace{4mm}\par\noindent \it 
Michael KUNZINGER~$^\dag$ and Roman O. POPOVYCH~$^\ddag$
\par\vspace{2mm}\par}

{\vspace{2mm}\par\it 
\noindent $^{\dag,\ddag}$Fakult\"at f\"ur Mathematik, Universit\"at Wien, Nordbergstra{\ss}e 15, A-1090 Wien, Austria
\par}

{\vspace{2mm}\par\noindent \it
$^{\ddag}$~Institute of Mathematics of NAS of Ukraine, 3 Tereshchenkivska Str., Kyiv-4, Ukraine
 \par}

{\vspace{2mm}\par\noindent $\phantom{^{\dag,\ddag}}$\rm E-mail: \it  $^\dag$michael.kunzinger@univie.ac.at, $^\ddag$rop@imath.kiev.ua
 \par}




{\vspace{5mm}\par\noindent\hspace*{5mm}\parbox{150mm}{\small
The notion of singular reduction operators, i.e., of singular operators of nonclassical (conditional) symmetry, 
of  partial differential equations in two independent variables is introduced. 
All possible reductions of these equations to first-order ODEs are are exhaustively described. 
As examples, properties of singular reduction operators of $(1+1)$-dimensional evolution and wave equations are studied. 
It is shown how to favourably enhance the derivation of
nonclassical symmetries for this class by an in-depth prior study of the corresponding singular vector fields. 
}\par\vspace{3mm}}

\section{Introduction}

Distinctions in kind between Lie symmetries and nonclassical symmetries became apparent already in the first presentation 
of nonclassical symmetries in~\cite{Bluman&Cole1969} 
by the example of the $(1+1)$-dimensional linear heat equation and a particular class of operators. 
In contrast to classical Lie symmetries (see, e.g.,~\cite{Ovsiannikov1982}), the system of determining equations on the coefficients 
of nonclassical symmetry operators of the heat equation was found to be nonlinear and less overdetermined,
and the set of such operators does not possess the structure of an algebra or even a vector space. 

Another difference appears in the procedure of deriving the determining equations. 
Namely, \emph{deriving systems of determining equations for nonclassical symmetries crucially 
depends on the interplay between the operators and the equations under consideration.} 
Thus, for the linear heat equation $u_t=u_{xx}$ 
the general form of nonclassical symmetry operators is $Q=\tau(t,x,u)\p_t+\xi(t,x,u)\p_x+\eta(t,x,u)\p_u$, 
where $(\tau,\xi)\ne(0,0)$,
and there are two essentially different cases of nonclassical symmetries: 
the \emph{regular} case $\tau\ne0$ and the \emph{singular} case $\tau=0$. 
The factorization up to nonvanishing functional multipliers gives the two respective cases 
for the further investigation: 1) $\tau=1$ and 2) $\tau=0$, $\xi=1$. 

The problem of determining the nonclassical symmetries of the linear heat equation was completely solved 
in \cite{Fushchych&Shtelen&Serov&Popovych1992}. 
In the regular case $\tau=1$, after partial integration of the corresponding determining equations, 
we obtain $\xi=g^1(t,x)$ and $\eta=g^2(t,x)u+g^3(t,x)$. 
The functions $g^1$, $g^2$ and $g^3$ satisfy 
a coupled nonlinear system of partial differential equations~\cite{Bluman&Cole1969}, 
which is linearized by a nonlocal transformation to a system of three uncoupled copies 
of the initial equation~\cite{Fushchych&Shtelen&Serov1993en,Fushchych&Shtelen&Serov&Popovych1992,Webb1990}. 
The underlying reason for this phenomenon lies in the interaction between the linearity and the evolution structure 
in the linear heat equation. 
Hence similar results can be obtained only for linear evolution equations
\cite{Fushchych&Popowych1994-1,Popovych1995,Popovych2008a} 
or related linearizable equations~\cite{Mansfield1999}.

The singular case $(\tau,\xi)=(0,1)$ was not considered in~\cite{Bluman&Cole1969}. 
In this case the system of determining equations for nonclassical symmetries 
consists of a single $(1+2)$-dimensional nonlinear evolution equation 
for the unknown function~$\eta$ and, therefore, is not overdetermined.
The determining equation is reduced by a nonlocal transformation to the initial equation with 
an additional implicit independent variable which can be assumed as a parameter~\cite{Fushchych&Shtelen&Serov&Popovych1992}. 
The linearity of the heat equation is inessential here. 
Hence after the case of linear evolutions equations~\cite{Fushchych&Popowych1994-1,Popovych1995} 
this result was extended to general $(1+1)$-dimensional evolution equations \cite{Zhdanov&Lahno1998}, 
multi-dimensional evolution equations~\cite{Popovych1998} 
and even systems of such equations~\cite{Vasilenko&Popovych1999}.
Moreover, it was proved~\cite{Popovych1998}, that, e.g., in the $(1+1)$-dimensional case 
there exists a one-to-one correspondence between one-parametric families of solutions of an evolution equation
and its reduction operators with $(\tau,\xi)=(0,1)$.

The above results raise a number of interesting questions, to wit:
What are possible causes for the existence of singular cases for reduction operators?
Is the conventional partition of sets of reduction operators with the conditions of vanishing and nonvanishing 
coefficients of operators universal or is it appropriate only for certain classes of differential equations, 
e.g., evolution equations?
Can partitions of sets of reduction operators, different from the conventional one, be useful? 
Does there exist an algorithmic way of singling out singular cases for reduction operators before deriving determining equations?
What properties of a partial differential equation and a subset of its reduction operators lead to 
a `no-go' situation (i.e., a single determining equation equivalent, in a certain sense, to the initial equation)?
What is the optimal way of obtaining the determining equation for nonclassical symmetries?
The purpose of the present paper is to answer these and other related questions. 

Algorithms for deriving the determining equations for nonclassical symmetries were discussed, e.g., in 
\cite{Bila&Niesen2004,Clarkson&Mansfield1993a} but the focus of these works was quite different.

The conditional invariance of a differential equation with respect to an operator is equivalent to 
any ansatz associated with this operator reducing the equation to a differential equation with one less 
independent variables~\cite{Zhdanov&Tsyfra&Popovych1999}. 
That is why we use the shorter and more natural term ``reduction operators'' 
instead of ``operators of conditional symmetry'' or ``operators of nonclassical symmetry''
and say that an operator reduces a differential equation 
if the equation is reduced by the corresponding ansatz. 
The direct method of reduction with ansatzes of a special form was first explicitly applied 
in~\cite{Clarkson&Kruskal1989} to the Boussinesq equation 
although reductions by non-Lie ansatzes were already discussed, e.g., in~\cite{Fushchych&Serov1983}. 
A connection between the reduction by generalized ansatzes 
and compatibility with respect to higher-order constraints was found in~\cite{Olver1994}. 

To clarify the main ideas of the proposed framework of singular reduction operators, 
in this first presentation of the subject 
we consider only the case of a single partial differential equation 
in two independent and one dependent variables and a single reduction operator. 
We note, however, that more general cases can be included and will be the subject of forthcoming papers.

Some of the main conclusions of the present paper are: 
\begin{itemize}\itemsep=0ex
\item
Singular cases of reduction operators of a partial differential equation 
are connected with the possibility of lowering the order of this equation on the manifolds determined 
by the corresponding invariant surface conditions in the appropriate jet space.  
Hence the first step of the procedure of finding nonclassical symmetries  
has to consist in studying singular modules of vector fields which lower the order of the equation. 
This step is entirely algorithmic, hence is especially suited to a direct implementation in symmetry-finding computer algebra programs. 
The structure of singular modules of vector fields has to be taken into account under splitting the set 
of reduction operators for factorization.
\item
The weak singularity co-order of a reduction operator~$Q$ coincides 
with the essential order of the corresponding reduced equation and the number of essential parameters in 
the family of $Q$-invariant solutions.
\item
If a single partial differential equation~$\mathcal L$ in two independent variables admits a first co-order singular module~$S$ of vector fields 
then it necessarily possesses first co-order singular reduction operators belonging to~$S$. 
The system of determining equations for such  operators consists of a single partial differential equation~$\mathrm{DE}$ 
in three independent variables of the same order as~$\mathcal L$. 
The equation~$\mathrm{DE}$ is reduced to~$\mathcal L$ by a nonlocal transformation.  
\end{itemize}

The paper is organized as follows: 
The main notions and statements on nonclassical symmetries are presented in Section~\ref{SectionOnDefOfRedOps}. 
Singular vector fields of differential functions and differential equations are defined and studied in 
Sections~\ref{SectionOnSingularVectorFieldsOfDiffFunctionsIn2DCase} 
and~\ref{SectionOnSingularVectorFieldsOfDiffEquationsIn2DCase}, respectively.  
Singular reduction operators of $(1+1)$-dimensional evolution and nonlinear wave equations 
are exhaustively investigated in Sections~\ref{SectionOnExampleOfEvolEqs} and~\ref{SectionOnExampleOfSimpleNWEs}.
It is shown that the conventional partition of sets of reduction operators is natural for evolution equations,
in contrast to the case of nonlinear wave equations. 
A connection between the singularity co-order of reduction operators and the number of parameters in the corresponding families of 
invariant solutions is established in Section~\ref{SectionOnReductionOperatorsAndParametricFamiliesOfSolutions}. 
The final Section~\ref{SectionOnReductionOperatorsOfSingularityCoOrder1} is devoted to 
first co-order singular reduction operators of general partial differential equations in 
two independent and one dependent variables.

\section{Reduction operators of differential equations}
\label{SectionOnDefOfRedOps}

Following~\cite{Fushchych&Tsyfra1987,Fushchych&Zhdanov1992,Popovych&Vaneeva&Ivanova2007,Zhdanov&Tsyfra&Popovych1999}, 
in this section we briefly collect the required notions and results on nonclassical (conditional) symmetries of 
differential equations. 
Also, we argue for the use of the name ``reduction operators'' 
instead of ``nonclassical (conditional) symmetry operators''. 
In accordance with the aims of this paper we restrict our considerations 
to the case of two independent variables and a single reduction operator. 
  
The set of (first-order) differential operators (or vector fields) of the general form
\[
Q=\xi^i(x,u)\p_i+\eta(x,u)\p_u, \quad (\xi^1,\xi^2)\not=(0,0),
\]
will be denoted by $\mathfrak Q$.
In what follows, 
$x$ denotes the pair of independent variables $(x_1,x_2)$ 
and $u$ is treated as the unknown function. 
The index $i$ runs from 1 to $2$,  
and we use the summation convention for repeated indices.
Subscripts of functions denote differentiation with respect to the corresponding variables, 
$\p_i=\p/\p x_i$ and $\p_u=\p/\p u$.
Any function is considered as its zero-order derivative.
All our considerations are carried out in the local setting.

Two differential operators $\widetilde Q$ and $Q$ are called \emph{equivalent} if they differ by a multiplier
which is a non-vanishing function of~$x$ and~$u$:
$\widetilde Q=\lambda Q$, where $\lambda=\lambda(x,u)$, $\lambda\not=0$.
The equivalence of operators will be denoted by $\widetilde Q\sim Q$.
Factoring~$\mathfrak Q$ with respect to this equivalence relation we arrive at~$\mathfrak Q_{\rm f}$.
Elements of~$\mathfrak Q_{\rm f}$ will be identified with their representatives in~$\mathfrak Q$.

The first-order differential function 
\[Q[u]:=\eta(x,u)-\xi^i(x,u)u_i\]
is called the {\it characteristic} of the operator~$Q$.
The characteristic PDE $Q[u]=0$ (also known as the \emph{invariant surface condition}) has
two functionally independent solutions $\zeta(x,u)$ and $\omega(x,u)$.
Therefore, the general solution of this equation can be implicitly represented in the form
$F(\zeta,\omega)=0$, where $F$ is an arbitrary function. 

A differential function $\Theta=\Theta[z]$ of the dependent variables $z=(z^1,\ldots,z^m)$ 
which in turn are functions of a tuple of independent variables $y=(y_1,\ldots,y_n)$ 
will be considered as a smooth function of $y$ and derivatives of~$z$ with respect to~$y$. 
The order $r=\ord \Theta$ of the differential function $\Theta$ equals the maximal order of derivatives involved in $\Theta$. 
More precisely, the differential function $\Theta$ is defined as a function on a subset of the jet space $J^r(y|z)$ \cite{Olver1993}.

The characteristic equations of equivalent operators have the same set of solutions.
Conversely, any family of two functionally independent functions of~$x$ and $u$
is a complete set of integrals of the characteristic equation of a differential operator.
Therefore, there exists a one-to-one correspondence between $\mathfrak Q_{\rm f}$
and the set of families of two functionally independent functions of~$x$ and $u$,
which is factorized with respect to the corresponding equivalence relation.
(Two families of the same number of functionally independent functions of the same arguments are considered equivalent
if any function from one of the families is functionally dependent on functions from the other family.)

Since $(\xi^1,\xi^2)\not=(0,0)$ we can assume without loss of generality that
$\zeta_u\not=0$ and $F_\zeta\not=0$ and resolve the equation~$F=0$ with respect to~$\zeta$: $\zeta=\varphi(\omega)$.
This implicit representation of the function~$u$ is called an \emph{ansatz} corresponding to the operator~$Q$.

Consider an $r$th order differential equation~$\mathcal L$ of the form~$L(x,u_{(r)})=0$
for the unknown function $u$ of two independent variables~$x=(x_1,x_2)$.
Here $L=L[u]=L(x,u_{(r)})$ is a fixed differential function of order~$r$ and 
$u_{(r)}$ denotes the set of all the derivatives of the function $u$ with respect to~$x$
of order not greater than~$r$, including $u$ as the derivative of order zero.
Within the local approach the equation~$\mathcal L$ is treated as an algebraic equation
in the jet space $J^r=J^r(x|u)$ of order $r$ and is identified with the manifold of its solutions in~$J^r$:
\begin{gather*}
\mathcal L=\{ (x,u_{(r)}) \in J^r\, |\, L(x,u_{(r)})=0\}.
\end{gather*}

Denote the manifold defined by the set of all the differential
consequences of the characteristic equation~$Q[u]=0$ in $J^r$
by $\mathcal Q_{(r)}$, i.e.,
\begin{gather*}
\mathcal Q_{(r)}=\{ (x,u_{(r)}) \in J^r\, |\,
D_1^{\alpha}D_2^{\beta}Q[u]=0, \
\alpha,\beta\in\mathbb{N}\cup\{0\},\ \alpha+\beta<r \},
\end{gather*}
where $D_1=\p_1+u_{\alpha+1,\beta}\p_{u_{\alpha\beta}}$ and
$D_2=\p_2+u_{\alpha,\beta+1}\p_{u_{\alpha\beta}}$ are the
operators of total differentiation with respect to the
variables~$x_1$ and~$x_2$, and the variable $u_{\alpha\beta}$ of the jet
space $J^r$ corresponds to the derivative
$\p^{\alpha+\beta}u/\p x_1^{\alpha}\p x_2^{\beta}$.

A precise and rigorous definition of nonclassical (or conditional) symmetry was first suggested in~\cite{Fushchych&Tsyfra1987} 
(see also~\cite{Fushchych&Zhdanov1992,Zhdanov&Tsyfra&Popovych1999}).

\begin{definition}\label{DefinitionOfCondSym}
The differential equation~$\mathcal L$ is called \emph{conditionally invariant} with respect to the operator $Q$ if
the relation $Q_{(r)}L(x,u_{(r)})\bigl|_{\mathcal L\cap\mathcal{Q}^{(r)}}=0$
holds, which is called the \emph{conditional invariance criterion}. 
Then $Q$ is called an operator of \emph{conditional symmetry} (or $Q$-conditional symmetry, nonclassical symmetry, etc.)
of the equation~$\mathcal L$.
\end{definition}

In Definition~\ref{DefinitionOfCondSym} the symbol $Q_{(r)}$ stands for the standard $r$th prolongation
of the operator~$Q$ \cite{Olver1993,Ovsiannikov1982}:
\[
Q_{(r)}=Q+\sum_{0<\alpha+\beta{}\leqslant  r} \eta^{\alpha\beta}\p_{u_{\alpha\beta}},
\quad
\eta^{\alpha\beta}:=D_1^\alpha D_2^\beta Q[u]+\xi^1 u_{\alpha+1,\beta}+\xi^2 u_{\alpha,\beta+1}.
\]

The equation~$\mathcal L$ is conditionally invariant with respect to~$Q$
if and only if the ansatz $\zeta=\varphi(\omega)$ constructed with~$Q$ reduces~$\mathcal L$
to an ordinary differential equation~$\check{\mathcal L}$: $\check L[\varphi]=0$~\cite{Zhdanov&Tsyfra&Popovych1999}. 
Namely, there exist differential functions 
$\smash{\check\lambda=\check\lambda[\varphi]}$ and $\smash{\check L=\check L[\varphi]}$ 
of an order not greater than~$r$ 
(i.e., functions of $\omega$ and derivatives of~$\varphi$ with respect to~$\omega$ up to order~$r$)
such that $L|_{u=\varphi(\omega)}=\check\lambda\check L$. 
The function~$\check\lambda$ does not vanish and may depend on~$\theta$ as a parameter, 
where the value $\theta=\theta(x,u)$ is functionally independent of~$\zeta$ and~$\omega$. 
The differential function~$\check L$ is assumed to be of minimal order~$\check r$ 
which is possibly reached up to the equivalence generated by nonvanishing multipliers. 
Then the reduced equation~$\check{\mathcal L}$ is of essential order~$\check r$. 

This is why we will also call operators of conditional symmetry {\it reduction operators} of~$\mathcal L$.

Another treatment of conditional invariance is that the system $\mathcal L\cap\mathcal Q_{(r)}$ is compatible 
in the sense of not involving any nontrivial differential consequences~\cite{Olver1994,Olver&Vorob'ev1996}. 

The property of conditional invariance is compatible with the equivalence relation on~$\mathfrak Q$ 
\cite{Fushchych&Zhdanov1992,Zhdanov&Tsyfra&Popovych1999}: 

\begin{lemma}\label{LemmaOnEquivFamiliesOfOperators}
If the equation~$\mathcal L$ is conditionally invariant with respect to the operator~$Q$
then it is conditionally invariant with respect to any operator which is equivalent to~$Q$.
\end{lemma}

The set of reduction operators of the equation~$\mathcal L$ is a
subset of $\mathfrak Q$ and so will be denoted by $\mathfrak Q(\mathcal L)$. 
In view of Lemma~\ref{LemmaOnEquivFamiliesOfOperators}, $Q\in \mathfrak Q(\mathcal L)$ and $\widetilde Q\sim Q$ imply
$\widetilde Q\in \mathfrak Q(\mathcal L)$, i.e., $\mathfrak Q(\mathcal L)$ is
closed under the equivalence relation on $\mathfrak Q$. 
Therefore, the factorization of $\mathfrak Q$ with respect to this equivalence
relation can be naturally restricted to~$\mathfrak Q(\mathcal L)$, 
resulting in the subset~$\mathfrak Q_{\rm f}(\mathcal L)$ of $\mathfrak Q_{\rm f}$. 
As in the whole set~$\mathfrak Q_{\rm f}$, we identify elements of~$\mathfrak Q_{\rm f}(\mathcal L)$ with their
representatives in~$\mathfrak Q(\mathcal L)$. 
In this approach the problem of completely describing all reduction operators for $\mathcal L$ is equivalent to finding~$\mathfrak Q_{\rm f}(\mathcal L)$.

The conditional invariance criterion admits the following useful reformulation~\cite{Zhdanov&Tsyfra&Popovych1999}. 

\begin{lemma}\label{LemmaOnReformulationOfCondInvCriterion}
Given a differential equation~$\mathcal L$: $L[u]=0$ of order~$r$ 
and differential functions $\tilde L[u]$ and $\lambda[u]\ne0$ of an order not greater than~$r$ 
such that $L|_{\mathcal Q_{(r)}}=\lambda\,\tilde L|_{\mathcal Q_{(r)}}$, 
an operator~$Q$ is a reduction operator of~$\mathcal L$ if and only if 
the relation $Q_{(\tilde r)}\tilde L\bigl|_{\tilde {\mathcal L}\cap\mathcal{Q}_{(\tilde r)}}=0$ holds, 
where $\tilde r=\ord\tilde L\leqslant r$ and the manifold~$\tilde {\mathcal L}$ is defined in $J^{\tilde r}$ 
by the equation $\tilde L[u]=0$. 
\end{lemma}

The classification of reduction operators can be considerably enhanced and simplified by considering 
Lie symmetry and equivalence transformations of (classes of) equations.

\begin{lemma}
Any point transformation of $x$ and $u$ induces a one-to-one mapping of~$\mathfrak Q$ into itself.
Namely, the transformation~$g$: $\tilde x_i=X^i(x,u)$, $\tilde u=U(x,u)$ generates
the mapping~\mbox{$g_*\colon \mathfrak Q\to\mathfrak Q$} such that
the operator~$Q$ is mapped to the operator
$g_*Q=\tilde\xi^i\p_{\tilde x_i}+\tilde\eta\p_{\tilde u}$, where
$\tilde\xi^i(\tilde x,\tilde u)=QX^i(x,u)$,
$\tilde\eta(\tilde x,\tilde u)=QU(x,u)$.
If~$Q'\sim Q$ then  $g_* Q'\sim g_* Q$.
Therefore, the corresponding factorized mapping~$g_{\rm f} \colon\mathfrak Q_{\rm f}\to\mathfrak Q_{\rm f}$ also
is well defined and bijective.
\end{lemma}

\begin{definition}[\cite{Popovych2000}]\label{DefinitionOfEquivInvFamiliesWrtGroup}
Differential operators $Q$ and $\widetilde Q$ are called equivalent
with respect to a group $G$ of point transformations if there exists $g\in G$
for which the operators $Q$ and $g_*\widetilde Q$ are equivalent.
We denote this equivalence by $Q\sim \widetilde Q \bmod G.$
\end{definition}

\begin{lemma}\label{LemmaOnInducedMapping}
Given any point transformation $g$ of an equation~$\mathcal L$ to an equation~$\tilde{\mathcal L}$,
$g_*$ maps~$\mathfrak Q(\mathcal L)$ to~$\mathfrak Q(\tilde{\mathcal L})$ bijectively.
The same is true for the factorized mapping $g_{\rm f}$ from $\mathfrak Q_{\rm f}(\mathcal L)$
to~$\mathfrak Q_{\rm f}(\tilde{\mathcal L})$.
\end{lemma}

\begin{corollary}\label{CorollaryOnEquivReductionOperatorWrtSymGroup}
Let $G$ be the point symmetry group of an equation~$\mathcal L.$ Then the equivalence of operators
with respect to the group $G$ generates equivalence relations in~$\mathfrak Q(\mathcal L)$
and in~$\mathfrak Q_{\rm f}(\mathcal L)$.
\end{corollary}

Consider the class~$\mathcal L|_{\mathcal S}$ of equations~$\mathcal L_\theta$: $L(x,u_{(r)},\theta)=0$
parameterized with the parameter-functions~$\theta=\theta(x,u_{(r)}).$
Here $L$ is a fixed function of $x$, $u_{(r)}$ and $\theta.$
The symbol~$\theta$~denotes the tuple of arbitrary (parametric) differential functions
$\theta(x,u_{(r)})=(\theta^1(x,u_{(r)}),\ldots,\theta^k(x,u_{(r)}))$
running through the set~${\mathcal S}$ of solutions of the system~$S(x,u_{(r)},\theta_{(q)}(x,u_{(r)}))=0$.
This system consists of differential equations on $\theta$,
where $x$ and $u_{(r)}$ play the role of independent variables
and $\theta_{(q)}$ stands for the set of all the derivatives of $\theta$ of order not greater than $q$.
In what follows we call the functions $\theta$ arbitrary elements.
Denote the point transformation group preserving the
form of the equations from~$\mathcal L|_{\mathcal S}$ by $G^\sim$.

Let $P$ denote the set of the pairs consisting of an
equation $\mathcal L_\theta$ from~$\mathcal L|_{\mathcal S}$ and an
operator~$Q$ from~$\mathfrak Q(\mathcal L_\theta)$. In view of
Lemma~\ref{LemmaOnInducedMapping}, the action of transformations from
the equivalence group~$G^\sim$ on $\mathcal L|_{\mathcal S}$ and
$\{\mathfrak Q(\mathcal L_{\theta})\,|\,\theta\in{\mathcal S}\}$
together with the pure equivalence relation of differential
operators naturally generates an equivalence relation on~$P$.

\begin{definition}\label{DefinitionOfEquivOfRedOperatorsWrtEquivGroup}
Let $\theta,\theta'\in{\mathcal S}$,
$Q\in\mathfrak Q(\mathcal L_\theta)$, $Q'\in\mathfrak Q(\mathcal L_{\theta'})$.
The pairs~$(\mathcal L_\theta,Q)$ and~$(\mathcal L_{\theta'},Q')$
are called {\em $G^\sim$-equivalent} if there exists $g\in G^\sim$
such that $g$ transforms the equation~$\mathcal L_\theta$ to the equation~$\mathcal L_{\theta'}$, and
$Q'\sim g_*Q$.
\end{definition}

The classification of reduction operators with respect to~$G^\sim$ will be understood as
the classification in~$P$ with respect to this equivalence relation, a
problem which can be investigated similar to the usual group classification in classes
of differential equations.
Namely, we construct firstly the reduction operators that are defined for all values of $\theta$.
Then we classify, with respect to $G^\sim$, the values of $\theta$ for which
the equation~$\mathcal L_\theta$ admits additional reduction operators.

\section{Singular vector fields of differential functions}\label{SectionOnSingularVectorFieldsOfDiffFunctionsIn2DCase}

Consider a vector field~$Q=\xi^i(x,u)\p_i+\eta(x,u)\p_u$ with $(\xi^1,\xi^2)\ne(0,0)$, defined in the space $(x,u)$, 
and a differential function~$L=L[u]$ of order $\ord L=r$ 
(i.e., a smooth function of $x=(x_1,x_2)$ and derivatives of~$u$ of orders up to~$r$).

\begin{definition}\label{DefinitionOfSingularVectorFild}
The vector field~$Q$ is called \emph{singular} for the differential function~$L$ 
if there exists a differential function $\tilde L=\tilde L[u]$ of an order less than~$r$
such that $L|_{\mathcal Q_{(r)}}=\tilde L|_{\mathcal Q_{(r)}}$. 
Otherwise $Q$ is called a \emph{regular} vector field for the differential function~$L$. 
If the minimal order of differential functions 
whose restrictions on $\mathcal Q_{(r)}$ coincide with $L|_{\mathcal Q_{(r)}}$ equals $k$ ($k<r$)
then the vector field~$Q$ is said to be of \emph{singularity co-order $k$} for the differential function~$L$. 
The vector field~$Q$ is called \emph{ultra-singular} for the differential function~$L$ if $L|_{\mathcal Q_{(r)}}\equiv0$. 
\end{definition}

For convenience, the singularity co-order of ultra-singular vector fields 
and the order of identically vanishing differential functions are defined to equal~$-1$. 
Regular vector fields for the differential function~$L$ are defined to have singularity co-order $r=\ord L$.
The singularity co-order of a vector field~$Q$ for a differential function~$L$ will be denoted by $\sco_LQ$. 

If $Q$ is a singular vector field for~$L$ then any vector field equivalent to~$Q$ is singular for~$L$  
with the same co-order of singularity.

A function $\tilde L$ satisfying the conditions of Definition~\ref{DefinitionOfSingularVectorFild} 
can be constructively found. 
Namely, without loss of generality we can suppose that the coefficient~$\xi^2$ of~$\p_2$ in~$Q$ is nonzero. 
Then any derivative of~$u$ of order not greater than~$r$ can be expressed, on the manifold~$\mathcal Q_{(r)}$, 
via derivatives of~$u$ with respect to $x_1$ only. 
For example, for the first- and second-order derivatives we have
\begin{equation}\begin{split}\label{EqSingularRedOpsExpressionsForU_2}
&u_2=\hat\eta-\hat\xi u_1,\\
&u_{12}=\hat\eta_1-\hat\xi_1u_1+\hat\eta_uu_1-\hat\xi_uu_1^2-\hat\xi u_{11},\\
&u_{22}=\hat\eta_2-\hat\xi_2u_1+(\hat\eta_u-\hat\xi_uu_1)(\hat\eta-\hat\xi u_1)
-\hat\xi(\hat\eta_1-\hat\xi_1u_1+\hat\eta_uu_1-\hat\xi_uu_1^2-\hat\xi u_{11}),
\end{split}\end{equation}
where $\hat\xi=\xi^1/\xi^2$ and $\hat\eta=\eta/\xi^2$. 
After substituting the expressions for the derivatives into~$L$, we obtain a differential function $\hat L$ 
depending only on $x$, $u$ and derivatives of~$u$ with respect to~$x_1$. 
We will call $\hat L$ a \emph{differential function associated with~$L$ on the manifold~$\smash{\mathcal Q_{(r)}}$}. 
The vector field~$Q$ is singular for the differential function~$L$ if and only if the order of $\hat L$ is less than~$r$. 
The co-order of singularity of~$Q$ equals the order of $\hat L$. 
The vector field~$Q$ is ultra-singular if and only if $\hat L\equiv0$. 
Therefore, testing that a vector field is singular for a differential function with two independent variables 
is realized in an entirely algorithmic procedure 
and can be easily included in existing programs for symbolic calculations of symmetries. 

Consider the two-dimensional module $\{Q^\theta=\theta^iQ^i\}$ of vector fields over the ring of smooth functions of $(x,u)$ 
generated by the vector fields $Q^i=\xi^{ij}(x,u)\p_j+\eta^i(x,u)\p_u$, where $\rank(\xi^{i1},\xi^{i2},\eta^i)=2$.
In the remainder of this section 
the parameter tuple $\theta=(\theta^1,\theta^2)$ runs through the set of pairs of smooth functions depending on $(x,u)$,
and $i$ and $j$ run from 1 to 2. 

\begin{definition}\label{DefinitionOfSingularModuleOfVectorFilds}
The module $\{Q^\theta\}$ is called \emph{singular} for the differential function~$L$ 
if for any $\theta$ with $(\theta^i\xi^{i1},\theta^i\xi^{i2})\ne(0,0)$ the vector field $Q^\theta$ is singular for~$L$.
The \emph{singularity co-order of the module} $\{Q^\theta\}$ coincides with 
the maximum of the singularity co-orders of its elements. 
\end{definition}

By a point transformation, one of the basis vector fields, e.g. $Q^2$, can be reduced to $\p_u$ 
(transforming $L$ simultaneously with $Q^1$ and $Q^2$.)
Then $(\xi^{11},\xi^{12})\ne(0,0)$, and up to permutation of independent variables 
we can assume $\xi^{12}\ne0$ and, therefore, set $\eta^1=0$ and $\xi^{12}=1$ by a change of basis. 
Any vector field from the module $\{Q^\theta\}$ with a nonzero value of~$\theta^1$ is equivalent to 
the vector field $Q^1+\zeta Q^2$, where $\zeta=\theta^2/\theta^1$. 
All the other vector fields from $\{Q^\theta\}$ 
(which have $\theta^1=0$ and, therefore, are equivalent to $\p_u$) can be neglected 
since each of them leads to the equation $\theta^2(x,u)=0$ which completely determines~$u$ and therefore, 
does not give an ansatz for~$u$.

This justifies why, up to point transformations, it suffices to study only singular sets of vector fields 
of the form $\{Q^\zeta=\xi\p_1+\p_2+\zeta\p_u\}$, with $\xi$ a fixed smooth function of $(x,u)$ 
and $\zeta$ running through all such functions. 
The latter form of singular sets of vector fields will be called \emph{reduced}. 

Further simplification depends on whether the module is closed under the Lie bracket. 
In case it is, it can be assumed to be generated by two commuting vector fields 
which can be simultaneously reduced by a point transformation to shift operators, e.g., $Q^1=\p_2$ and $Q^2=\p_u$. 
Thus in the reduced form $\xi$ can be put to $0$. 
If the module is not closed under the Lie bracket, we have $\xi_u\not=0$ in the reduced form. 
After the point transformation $\tilde x_i=x_i$ and  $\tilde u=\xi$ and a change of basis, 
we obtain the basis $\tilde Q^1=\tilde u\p_{\tilde 1}+\p_{\tilde 2}$ and $\tilde Q^2=\p_{\tilde u}$. 
Hence: 

\begin{proposition}\label{PropositionOnBasesOf2DModulesOf2DVectorFields}
In any two-dimensional module of vector fields in the space of three variables $(x_1,x_2,u)$, 
any basis vector fields $Q^1$ and $Q^2$ can be locally reduced, by point transformations, to the form 
$Q^1=\p_2$ (resp. $Q^1=u\p_1+\p_2$) and $Q^2=\p_u$ if the module is closed (resp. not closed) 
with respect to the Lie bracket of vector fields.
\end{proposition}

\begin{theorem}\label{TheoremOn2DSingularVectorFieldsOfCoOrderK}
A differential function~$L$ with one dependent and two independent variables possesses
a $k$th co-order singular two-dimensional module of vector fields if and only if 
it can be represented, up to point transformations, in the form 
\begin{equation}\label{Eq2DWithSingularVectorFieldsOfCoOrderK}
L=\check L(x,\Omega_{r,k}),
\end{equation}
where $\Omega_{r,k}=\bigl(\omega_\alpha=D_1^{\alpha_1}(\xi D_1+D_2)^{\alpha_2}u,\alpha_1\leqslant k,\alpha_1+\alpha_2\leqslant r\bigr)$, 
$\xi\in\{0,u\}$, 
and $\check L_{\omega_\alpha}\ne0$ for some $\omega_\alpha$ with $\alpha_1=k$.
\end{theorem}

\begin{proof}
Suppose that a differential function~$L$ possesses a $k$th co-order singular two-dimen\-sional module of vector fields 
$\{Q^\theta=\theta^iQ^i\}$. 
By a point transformation and a change of basis, we represent the basis elements in the reduced form 
$Q^1=\xi\p_1+\p_2$ and $Q^2=\p_u$, where $\xi\in\{0,u\}$, 
and choose the subset $\{Q^\zeta=\xi\p_1+\p_2+\zeta\p_u\}$ in $\{Q^\theta\}$, where
$\zeta$ runs through the set of smooth functions of $(x,u)$. 
The initial differential function also will be changed by these transformations but throughout we will 
use the old notations for all new values.

We fix an arbitrary point $\smash{z^0=(x^0,u_{(r)}^0)\in J^r}$ 
and consider the vector fields from $\{Q^\zeta\}$ for which $\smash{z^0\in\mathcal Q^\zeta_{(r)}}$. 
This condition implies that the values of the derivatives of~$\zeta$ with respect to only~$x_1$ and $x_2$ 
in the point $(x^0,u^0)$ are expressed via $\smash{u_{(r)}^0}$ and values of derivatives of~$\zeta$ in $(x^0,u^0)$, 
containing differentiation with respect to~$u$. 
The latter values are not constrained.

We introduce the new coordinates $\{x_i,\omega_\alpha=D_1^{\alpha_1}(\xi D_1+D_2)^{\alpha_2}u,|\alpha|\leqslant r\}$ 
in $J^r$ instead of the standard ones $\{x_i,u_\alpha,|\alpha|\leqslant r\}$. 
This is a valid change of coordinates since the Jacobian matrix $(\p\omega_\alpha/\p u_{\alpha'})$ is nondegenerate. 
Indeed, it is a triangular matrix with all diagonal entries equal to $1$ if the following order of multi-indices is implemented:  
$\alpha<\beta:\Leftrightarrow|\alpha|<|\beta|\vee(|\alpha|=|\beta|\wedge\alpha_2<\beta_2)$. 
Note that $\omega_\alpha=D_1^{\alpha_1}(\xi D_1+D_2)^{\alpha_2}u=D_1^{\alpha_1}(Q^\zeta)^{\alpha_2}u$ on $\smash{\mathcal Q^\zeta_{(r)}}$.

Consider the differential function $\hat L$ obtained from~$L$ by the above procedure of excluding, 
on the manifold $\smash{\mathcal Q^\zeta_{(r)}}$, the derivatives of~$u$ involving differentiations with respect to~$x_2$ 
(see~\eqref{EqSingularRedOpsExpressionsForU_2}). 
Since $Q^\zeta$ is a $k$th co-order singular vector field for~$L$, 
the function $\hat L$ does not depend on the derivatives $u_{(\kappa,0)}$, $\kappa=k+1,\dots,r$. 
We use this condition step-by-step, starting from the greatest value of~$\kappa$ 
and re-writing the derivatives in the new coordinates of~$J^r$ and in terms of~$L$. 

Thus, in the new coordinates the equation $\hat L_{u_{(r,0)}}(z^0)=0$ has the form $L_{\omega_{(r,0)}}(z^0)=0$. 
This completes the first step.
Then in the second step the equation $\hat L_{u_{(r-1,0)}}(z^0)=0$ implies that 
\[
L_{\omega_{(r-1,0)}}(z^0)+L_{\omega_{(r-1,1)}}(z^0)\zeta_u(x^0,u^0)=0.
\]
We split with respect to the value $\zeta_u(x^0,u^0)$ since it is unconstrained.
As a result, we obtain the equations $L_{\omega_{(r-1,0)}}(z^0)=0$ and $L_{\omega_{(r-1,1)}}(z^0)=0$.

Iterating this procedure, before the $\mu$th step, $\mu\in\{1,\dots,r-k\}$, we derive the equations 
$L_{\omega_{(r-\mu',\nu)}}(z^0)=0$, $\mu'=0,\dots,\mu-2$, $\nu=0,\dots,\mu'$.
Then the equation $\hat L_{u_{(r-\mu+1,0)}}(z^0)=0$ implies that 
\[
\sum_{\nu=0}^{\mu-1}
L_{\omega_{(r-\mu+1,\nu)}}(z^0)\bigl(\p_u(Q^\zeta)^\nu u\bigr)\big|_{(x,u)=(x^0,u^0)}=0.
\]
The values $\p_u^{\nu+1}\zeta(x^0,u^0)$, $\nu=0,\dots,\mu-1$, are unconstrained.
Splitting with respect to them, which is equivalent to splitting with respect to 
$\smash{\bigl(\p_u(Q^\zeta)^\nu u\bigr)\big|_{(x,u)=(x^0,u^0)}}$, $\nu=0,\dots,\mu-1$,
gives the equations $L_{\omega_{(r-\mu+1,\nu)}}(z^0)=0$, $\nu=0,\dots,\mu-1$.

Finally, after the $(r-k)$th step we derive the system 
$L_{\omega_{(r-\mu',\nu)}}(z^0)=0$, $\mu'=0,\dots,r-k+1$, $\nu=0,\dots,\mu'$, 
which implies condition~\eqref{Eq2DWithSingularVectorFieldsOfCoOrderK}.

Conversely, let an $r$th order differential function~$L$ be of the form~\eqref{Eq2DWithSingularVectorFieldsOfCoOrderK} 
(after a point transformation). 
For an arbitrary smooth function~$\zeta=\zeta(x,u)$ we consider 
the vector field $Q^\zeta=\xi\p_1+\p_2+\zeta\p_u$ and 
the differential function $\tilde L=\check L(x,\tilde\Omega_{r,k})$
where \[\tilde\Omega_{r,k}=\bigl(\omega_\alpha=D_1^{\alpha_1}(Q^\zeta)^{\alpha_2}u,\alpha_1\leqslant k,\alpha_1+\alpha_2\leqslant r\bigr).\]
Then $\ord \tilde L=k$ and \[L|_{\mathcal Q^\zeta_{(r)}}=\tilde L|_{\mathcal Q^\zeta_{(r)}},\] 
i.e., $\{Q^\zeta=Q^1+\zeta Q^2\}$, where $Q^1=\xi\p_1+\p_2$, $Q^2=\p_u$
and $\zeta$ runs through the set of smooth functions of $(x,u)$, 
is a $k$th co-order singular set for the differential function~$L$ in the new variables. 
We complete the set by the vector fields equivalent to its elements or $\p_u$ and return to the old variables. 
As a result, for the differential function~$L$ we construct 
a $k$th co-order singular two-dimensional module of vector fields $\{Q^\theta=\theta^iQ^i\}$.
\end{proof}

\begin{corollary}\label{CorollaryOn2DCommutingSingularVectorFieldsOfCo-orderK}
A differential function with one dependent and two independent variables admits 
a $k$th co-order singular two-dimensional module generated by commuting vector fields 
if and only if it can be reduced by a point transformation of the variables to a differential function
in which all differentiations with respect to one of the independent variables are of order $\le k$. 
\end{corollary}

\begin{corollary}\label{CorollaryOn2DUltraSingularVectorFields}
Any differential function with one dependent and two independent variables (not identically vanishing) admits 
no ultra-singular two-dimensional module of singular vector fields. 
\end{corollary}

\begin{note}\label{NoteOnSingularVectorFieldsOfCo-ordersLessThanModuleCo-order}
It is obvious that a singular module may contain vector fields whose singularity co-orders 
are less than the singularity co-order of the whole module. 
Suppose that $\{Q^\zeta=\xi\p_1+\p_2+\zeta\p_u\}$ is a singular set of vector fields for a differential function~$L$, 
and its singularity co-order equals~$k$.
Then the values of~$\zeta$ for which $\sco_LQ^\zeta<k$ are solutions of the equation 
\[
\sum_{\nu=0}^{r-k}\check L_{\omega_{(k,\nu)}}(x,\tilde\Omega_{r,k})\bigl(\p_u(Q^\zeta)^\nu u\bigr)=0, 
\]
where $\tilde\Omega_{r,k}=\bigl(D_1^{\alpha_1}(Q^\zeta)^{\alpha_2}u,\alpha_1\leqslant k,\alpha_1+\alpha_2\leqslant r\bigr)$ 
and $\check L$ is defined in Theorem~\ref{TheoremOn2DSingularVectorFieldsOfCoOrderK}. 
In other words, the regular values of~$\zeta$ associated with the vector fields of the maximal singularity co-order~$k$ in $\{Q^\zeta\}$
satisfy the inequality 
\[
\sum_{\nu=0}^{r-k}\check L_{\omega_{(k,\nu)}}(x,\tilde\Omega_{r,k})\bigl(\p_u(Q^\zeta)^\nu u\bigr)\ne0.
\]
\end{note}

\section{Singular vector fields of differential equations}\label{SectionOnSingularVectorFieldsOfDiffEquationsIn2DCase}

We will say that a vector field~$Q$ is \emph{(strongly) singular for a differential equation}~$\mathcal L$ 
if it is singular for the differential function~$L[u]$ which is the left hand side of the canonical representation $L[u]=0$ 
of the equation~$\mathcal L$. 
Usually we will omit the attribute ``strongly''.

Since left hand sides of differential equations are defined up to multipliers which are nonvanishing differential functions,
the conditions from Definition \ref{DefinitionOfSingularVectorFild} can be weakened 
when considering differential equations.

\begin{definition}\label{DefinitionOfWeaklySingularVectorFild}
A vector field~$Q$ is called \emph{weakly singular} for the differential equation~$\mathcal L$: $L[u]=0$ 
if there exist a differential function $\tilde L=\tilde L[u]$ of an order less than~$r$ 
and a nonvanishing differential function $\lambda=\lambda[u]$ of an order not greater than~$r$ 
such that $L|_{\mathcal Q_{(r)}}=\lambda\,\tilde L|_{\mathcal Q_{(r)}}$. 
Otherwise $Q$ is called a \emph{weakly regular} vector field for the differential equation~$\mathcal L$. 
If the minimal order of differential functions 
whose restrictions on $\mathcal Q_{(r)}$ coincide, up to nonvanishing functional multipliers, with $L|_{\mathcal Q_{(r)}}$ is equal to $k$ ($k<r$)
then the vector field~$Q$ is said to be \emph{weakly singular of co-order $k$} for the differential equation~$\mathcal L$. 
\end{definition}

The notions of ultra-singularity in the weak and the strong sense coincide. 
Analogous to the case of strong regularity, weakly regular vector fields for the differential equation~$\mathcal L$ 
are defined to have weak singularity co-order $r=\ord L$. 
The weak singularity co-order of a vector field~$Q$ for an equation~$\mathcal L$ will be denoted by $\wsco_{\mathcal L}Q$. 

Note that strong singularity implies weak singularity and hence weak regularity implies strong regularity.
It is obvious that the weak singularity co-order is never greater and may be less than the strong singularity co-order. 
In particular, strongly regular vector fields may be singular in the weak sense. 
For example, the equation $u_{ttt}=e^{u_{xx}}(u_x+u)$ possesses the singular vector field $\p_t$ 
whose strong and weak singularity co-order equal~2 and~1, respectively. 
The same vector field $\p_t$ is strongly regular and is of weak singularity co-order~1 for the equation $u_t=e^{u_{xx}}(u_x+u)$.

If $Q$ is a weakly singular vector field for~$\mathcal L$ then any vector field equivalent to~$Q$ is weakly singular for~$\mathcal L$  
with the same co-order of weak singularity. 

Weakly singular vector fields are related to characteristic directions
(cf.~\cite{Olver1993} concerning characteristic directions and characteristic hypersurfaces):
Given a vector field~$Q=\xi^i(x,u)\p_i+\eta(x,u)\p_u$ weakly singular for a differential equation~$\mathcal L$,
in each point of the manifold~$\mathcal L$ the vector $(\xi^1,\xi^2)$ is orthogonal to a characteristic direction 
of the equation~$\mathcal L$ in this point. 

Let $\hat L$ be a differential function associated with~$L$ on the manifold~$\mathcal Q_{(r)}$, namely, 
obtained from~$L$ via excluding those derivatives of~$u$ which contain differentiations with respect to~$x_2$ in view of equations defining~$\mathcal Q_{(r)}$.
Suppose additionally that $\hat L$ is of maximal rank in the derivative~$u$ of the highest order~$k$ appearing in this differential function, i.e.,
$\smash{\hat L_{u_{(k,0)}}\ne0}$ on the solution manifold of the equation~$\hat L=0$.
Then the weak singularity co-order of~$Q$ for the equation~$\mathcal L$: $L=0$ equals the order~$k$ of $\hat L$ 
and, therefore, the strong singularity co-order of~$Q$.
Hence in this case
testing that a vector field is weakly singular for a partial differential equation with two independent variables 
can be implemented via an entirely algorithmic procedure.

\begin{theorem}\label{TheoremOn2DWeakSingularVectorFieldsOfCo-orderK}
An $r$th order differential equation~$\mathcal L$: $L[u]=0$ of maximal rank with one dependent and two independent variables possesses
a $k$th co-order weakly singular two-dimensional module of vector fields if and only if 
$L$ can be represented, up to point transformations, in the form 
\begin{equation}\label{Eq2DWithWeakSingularVectorFieldsOfCoOrderK}
L=\Lambda[u]\check L(x,\Omega_{r,k}),
\end{equation}
where $\Lambda$ is a nonvanishing differential function of order not greater than~$r$, 
$\Omega_{r,k}=\bigl(\omega_\alpha=D_1^{\alpha_1}(\xi D_1+D_2)^{\alpha_2}u,\alpha_1\leqslant k,\alpha_1+\alpha_2\leqslant r\bigr)$, 
$\xi\in\{0,u\}$, 
and $\check L_{\omega_\alpha}\ne0$ for some $\omega_\alpha$ with $\alpha_1=k$.
\end{theorem}

\begin{proof}
We will freely use the notations and definitions from the proof of Theorem~\ref{TheoremOn2DSingularVectorFieldsOfCoOrderK}.

Suppose at first that a differential equation~$\mathcal L$: $L[u]=0$ is of maximal rank and admits 
a $k$th co-order weakly singular two-dimensional module of vector fields. 
Up to point transformations and changes of module basis, we may consider only 
a set $\{Q^\zeta=\xi\p_1+\p_2+\zeta\p_u\}$ of singular vector fields in reduced form. 

We fix an arbitrary point $\smash{z^0=(x^0,u_{(r)}^0)\in\mathcal L\subset J^r}$ 
and choose the vector fields from $\{Q^\zeta\}$ for which $\smash{z^0\in\mathcal Q^\zeta_{(r)}}$. 
This condition implies that the values of derivatives of~$\zeta$ with respect to only~$x_1$ and $x_2$ 
in the point $(x^0,u^0)$ are expressed via $\smash{u_{(r)}^0}$ and values of derivatives of~$\zeta$ in $(x^0,u^0)$, 
containing differentiation with respect to~$u$. 
The latter values are not constrained.

The differential function $\hat L$ is obtained from~$L$ by excluding, 
on the manifold $\smash{\mathcal Q^\zeta_{(r)}}$, derivatives of~$u$ involving differentiations with respect to~$x_2$ 
(see~\eqref{EqSingularRedOpsExpressionsForU_2}). 
$k$th co-order weak singularity of $Q^\zeta$ for~$L$ leads to  
$\hat L_{u_{(\kappa,0)}}(z_0)=0$, $\kappa=k+1,\dots,r$. 
We use this condition step-by-step as in the proof of Theorem~\ref{TheoremOn2DSingularVectorFieldsOfCoOrderK}, 
starting from the greatest value of~$\kappa$ and re-writing the derivatives in the new coordinates 
$\{x_i,\omega_\alpha=D_1^{\alpha_1}(\xi D_1+D_2)^{\alpha_2}u,|\alpha|\leqslant r\}$ of~$J^r$ and in terms of~$L$. 
Therefore, 
\[L_{\omega_{(r-\mu',\nu)}}(z^0)=0,\quad \mu'=0,\dots,r-k+1,\quad \nu=0,\dots,\mu',\]
which is satisfied for any $z^0\in\mathcal L$. 
Applying the Hadamard lemma to each of these equations and then simultaneously integrating them,  
we obtain~\eqref{Eq2DWithWeakSingularVectorFieldsOfCoOrderK} (cf.\ the proof of Theorem~1 in \cite{Zhdanov&Tsyfra&Popovych1999}).

Conversely, let an $r$th order differential function~$L$ be of the form~\eqref{Eq2DWithWeakSingularVectorFieldsOfCoOrderK} 
(after a point transformation). 
For an arbitrary smooth function~$\zeta=\zeta(x,u)$ we consider 
the vector field $Q^\zeta=\xi\p_1+\p_2+\zeta\p_u$ and 
the differential function $\tilde L=\check L(x,\tilde\Omega_{r,k})$, 
where \[\tilde\Omega_{r,k}=\bigl(\omega_\alpha=D_1^{\alpha_1}(Q^\zeta)^{\alpha_2}u,\alpha_1\leqslant k,\alpha_1+\alpha_2\leqslant r\bigr).\]
Then $\ord \tilde L=k$ and $L|_{\mathcal Q^\zeta_{(r)}}=\Lambda\tilde L|_{\mathcal Q^\zeta_{(r)}}$, 
i.e., $\{Q^\zeta=Q^1+\zeta Q^2\}$, where $Q^1=\xi\p_1+\p_2$, $Q^2=\p_u$
and $\zeta$ runs through the set of smooth functions of $(x,u)$, 
is a $k$th co-order weakly singular set for the differential equation~$\mathcal L$ in the new variables. 
We complete the set by the vector fields equivalent to its elements or $\p_u$ and return to the old variables, 
thereby constructing 
a $k$th co-order weakly singular two-dimensional module of vector fields $\{Q^\theta=\theta^iQ^i\}$
for the differential equation~$\mathcal L$.
\end{proof}

\begin{corollary}\label{CorollaryUselessnessOfWeaklySingularFamiliesOfVectorFields}
A differential equation~$\mathcal L$: $L[u]=0$ of maximal rank with one dependent and two independent variables possesses
a $k$th co-order weakly singular two-dimensional module of vector fields if and only if 
this module is $k$th co-order strongly singular for~$\mathcal L$ 
(possibly in a representation differing from $L[u]=0$ in multiplication by a nonvanishing differential function~of~$u$).
\end{corollary}

\begin{definition}\label{DefinitionOfSingularRedOps}
A vector field~$Q$ is called a \emph{singular reduction operator} of a differential equation~$\mathcal L$
if $Q$ is both a reduction operator of~$\mathcal L$ and a weakly singular vector field of~$\mathcal L$.
\end{definition}

\section{Example: evolution equations}\label{SectionOnExampleOfEvolEqs}

In this section we investigate singular reduction operators of $(1+1)$-dimensional evolution equations of the form 
\begin{equation}\label{EqGenEvolEq}
u_t=H(t,x,u_{(r,x)}),
\end{equation}
where $r>1$, $u_0:=u$, $u_k=\p^ku/\p x^k$, $u_{(r,x)}=(u_0,u_1,\dots,u_r)$ and  $H_{u_r}\ne0$.
(We revert to the notation $t$ and $x$ for $x_1$ and $x_2$, respectively, and change the notations of the corresponding derivatives.)
Evolution equations are quite specific from the point of view of singular vector fields and singular reduction operators. 

\begin{proposition}\label{PropositionOnSingularVectorFieldsOfSimpleNWEs}
A vector field~$Q=\tau(t,x,u)\p_t+\xi(t,x,u)\p_x+\eta(t,x,u)\p_u$ is singular for the differential function $L=u_t-H(t,x,u_{(r,x)})$ 
of order $r>1$ if and only if $\tau=0$. 
The co-order of singularity of any singular vector field for any such differential function equals~1.
\end{proposition}

\begin{proof}
Suppose that $\tau\ne0$. 
Excluding the derivative $u_t$ from~$L$ according to the equation $u_t=\eta/\tau-\xi u_x/\tau$ 
results in a differential function $\tilde L=\eta/\tau-\xi u_x/\tau-H(t,x,u_{(r,x)})$. 
Since $\ord \tilde L=r=\ord L$, the vector field $Q$ is not singular in this case. 
Therefore, for the vector field $Q$ to be singular, the coefficient~$\tau$ has to vanish. 

If $\tau=0$ and therefore $\xi\ne0$, all the derivatives $u_k$, $k=1,\dots,r$, can be expressed, on the manifold $\mathcal Q_{(r)}$ via $t$, $x$ and~$u$:
$u_k=(\p_x+\zeta\p_u)^{k-1}\zeta$, $k=1,\dots,r$, where $\zeta=\eta/\xi$. 
Using these expressions for excluding the derivatives $u_k$, $k=1,\dots,r$ from~$L$, we obtain the differential function 
\[
\tilde L=u_t-\tilde H(t,x,u), \quad \tilde H:=H(t,x,u,\zeta,\zeta_x+\zeta\zeta_u,\dots,(\p_x+\zeta\p_u)^{r-1}\zeta),
\]
whose order equals~1. 
Hence the vector field $Q$ is singular for the differential function $L$, and its singularity co-order equals~1.
\end{proof}

\begin{corollary}\label{CorollaryOSingularVectorFieldsOfEvolEqs}
For any $(1+1)$-dimensional evolution equation, 
the corresponding differential function possesses exactly one set of singular vector fields in the reduced form, 
namely, $S=\{\p_x+\zeta(x,u)\p_u\}$. 
The singularity co-order of~$S$ equals~1.
\end{corollary}

It is obvious that under the condition $H_{u_r}\ne0$ 
a vector field is singular for the differential function $u_t-H(t,x,u_{(r,x)})$ 
if and only if it is weakly singular for the differential equation $u_t=H(t,x,u_{(r,x)})$.
Hence we do not distinguish between strong and weak singularity 
(cf. Corollary~\ref{CorollaryUselessnessOfWeaklySingularFamiliesOfVectorFields}). 

The vector fields~$\p_2$ and~$\p_u$ generating the singular module associated with~$S$ commute
and the differential function~$L$ contains only first order differentiation with respect to~$t$ 
(namely, in the form of the derivative~$u_t$). 
This perfectly agrees with Corollary~\ref{CorollaryOn2DCommutingSingularVectorFieldsOfCo-orderK}.

We fix an arbitrary equation~$\mathcal L$ of the form~\eqref{EqGenEvolEq} and denote by $\mathcal Q_0(\mathcal L)$ 
the set of reduction operators of~$\mathcal L$, belonging to~$S$. 
For the equation~$\mathcal L$ and $Q\in\mathcal Q_0(\mathcal L)$, the conditional invariance criterion implies only the single $r$th order equation
\[
\zeta_t+\zeta_u\tilde H=\tilde H_x+\zeta\tilde H_u, \quad \tilde H:=H(t,x,u,\zeta,\zeta_x+\zeta\zeta_u,\dots,(\p_x+\zeta\p_u)^{r-1}\zeta),
\]
with respect to the single unknown function~$\zeta$ with three independent variables~$t$, $x$ and~$u$, which we will denote by ${\rm DE}_0(\mathcal L)$.
In other words, the system of determining equations in this case consists of the single equation ${\rm DE}_0(\mathcal L)$ 
and, therefore, is not overdetermined.
${\rm DE}_0(\mathcal L)$ is the compatibility condition 
of the equations $u_x=\zeta$ and $\mathcal L$. 

\begin{theorem}\label{TheoremUnitedOnSetsOfSolutionsAndReductionOperatorsWithTau0OfSystemsOfEqlsp}
Up to the equivalences of operators and solution families, for any equation of form~\eqref{EqGenEvolEq} 
there exists a one-to-one correspondence between one-parametric families of its solutions 
and reduction operators with zero coefficients of $\p_t$.
Namely, each operator of this kind corresponds to
the family of solutions which are invariant with respect to this operator. 
The problems of the construction of all one-parametric solution families of equation~\eqref{EqGenEvolEq} 
and the exhaustive description of its reduction operators with zero coefficients of $\p_t$
are completely equivalent.
\end{theorem}

\begin{proof}
Let $\mathcal L$ be an equation from class~\eqref{EqGenEvolEq} and $Q=\p_x+\zeta\p_u\in \mathcal Q_0(\mathcal L)$, i.e., 
the coefficient $\zeta=\zeta(t,x,u)$ satisfies the equation ${\rm DE}_0(\mathcal L)$. 
An ansatz constructed with $Q$ has the form $u=f(t,x,\varphi(\omega))$, 
where $f=f(t,x,\varphi)$ is a given function, $f_\varphi\ne0$, 
$\varphi=\varphi(\omega)$ is the new unknown function and $\omega=t$ is the invariant independent variable.
This ansatz reduces $\mathcal L$ to a first-order ordinary differential equation $\mathcal L'$ in~$\varphi$, 
solvable with respect to~$\varphi'$. 
The general solution of the reduced equation~$\mathcal L'$ can be represented in the form 
$\varphi=\varphi(\omega,\varkappa)$, where $\varphi_\varkappa\ne0$ and $\varkappa$ is an arbitrary constant. 
Substituting this solution into the ansatz results in  
the one-parametric family~$\mathcal F$ of solutions $u=\tilde f(t,x,\varkappa)$ of~$\mathcal L$ 
with $\tilde f=f(t,x,\varphi(t,\varkappa))$. 
Expressing the parameter~$\varkappa$ from the equality  $u=\tilde f(t,x,\varkappa)$, 
we obtain that $\varkappa=\Phi(t,x,u)$, where $\Phi_u\ne0$.  
Then $\zeta=u_x=-\Phi_x/\Phi_u$ for any $u\in\mathcal F$, i.e., for any admissible value of $(t,x,\varkappa)$. 
This implies that $\zeta=-\Phi_x/\Phi_u$ for any admissible value of $(t,x,u)$.

Conversely, 
suppose that $\mathcal F=\{u=f(t,x,\varkappa)\}$ is a one-parametric family of solutions of~$\mathcal L$. 
The derivative $f_\varkappa$ is nonzero since the parameter $\varkappa$ is essential. 
We express $\varkappa$ from the equality $u=f(t,x,\varkappa)$: $\varkappa=\Phi(t,x,u)$ for some function $\Phi=\Phi(t,x,u)$ 
with $\Phi_u\ne0$. 
Consider the operator $Q=\p_x+\zeta\p_u$, where the coefficient $\zeta=\zeta(t,x,u)$ is defined by $\zeta=-\Phi_x/\Phi_u$. 
$Q[u]=0$ for any $u\in\mathcal F$. 
The ansatz $u=f(t,x,\varphi(\omega))$, where $\omega=t$, associated with~$Q$, 
reduces~$\mathcal L$ to the equation $\varphi_\omega=0$.
Therefore \cite{Zhdanov&Tsyfra&Popovych1999}, $Q\in\mathcal Q_0(\mathcal L)$ and hence
the function $\zeta$ satisfies ${\rm DE}_0(\mathcal L)$. 
\end{proof}

\begin{corollary}\label{CorollaryOnLinearizationOfDetEqs0ForRedOpsOfLPEs}
The nonlinear $(1+2)$-dimensional evolution equation ${\rm DE}_0(\mathcal L)$ is reduced by the composition of the nonlocal substitution 
$\zeta=-\Phi_x/\Phi_u$, where $\Phi$ is a function of $(t,x,u)$, and the hodograph transformation 
\begin{gather*}
\mbox{the new independent variables:}\qquad\tilde t=t, \quad \tilde x=x, \quad \varkappa=\Phi, 
\\ 
\lefteqn{\mbox{the new dependent variable:}}\phantom{\mbox{the new independent variables:}\qquad }\tilde u=u 
\end{gather*}
to the initial equation $\mathcal L$ in the function $\tilde u=\tilde u(\tilde t,\tilde x,\varkappa)$ 
with $\varkappa$ playing the role of a parameter.
\end{corollary}

\begin{note}\label{NoteOnEquiv1ParamFamiliesOfSolutions}
One-parametric families  $u=f(t,x,\varkappa)$ and $u=\tilde f(t,x,\tilde\varkappa)$ 
are defined to be equivalent if they consist of the same functions and differ only by parameterizations, i.e., 
if there exists a function $\zeta=\zeta(\varkappa)$ such that 
$\zeta_\varkappa\ne0$ and $\tilde f(t,x,\zeta(\varkappa))=f(t,x,\varkappa)$.  
Equivalent one-parametric families of solutions are associated with the same operator 
from $\mathcal Q_0(\mathcal L)$ and have to be identified. 
\end{note}

\begin{note}
The triviality of the above ansatz and the reduced equation results from
the above special representation for the solutions of the determining equation.  
Under this approach difficulties in the construction of ansatzes and the integration of the reduced equations are replaced 
by difficulties in obtaining the representation for the coefficients of the reduction operators.
\end{note}

The above consideration shows that for any evolution equation~$\mathcal L$ 
the conventional partition of the set~$\mathfrak Q(\mathcal L)$ of its reduction operators with the conditions $\tau\ne0$ and $\tau=0$ 
is natural since it coincides with the partition of~$\mathfrak Q(\mathcal L)$ into the singular and regular subsets. 
\emph{This is a specific property of evolution equations which does not hold for general partial differential equations 
in two independent variables.} 
After factorizing the subsets of~$\mathfrak Q(\mathcal L)$ with respect to the usual equivalence relation of reduction operators, 
we obtain two different cases of inequivalent reduction operators (the regular case $\tau=1$ and the singular case $\tau=0$ and $\xi=1$), 
which have to be studied separately. 

Singular reduction operators of~$\mathcal L$ are described in a unified `no-go' way. 
All singular reduction operators of~$\mathcal L$ have the same singularity co-order equal to~1 
and hence reduce~$\mathcal L$ to first order ordinary differential equations. 
The coincidence of the singularity co-orders guarantees the existence of a bijection between 
the set of singular reduction operators of~$\mathcal L$ and the set of one-parametric families of its solutions 
(up to the natural equivalence relations in these sets). 
As a result, in the case $\tau=0$ and $\xi=1$ the determining equation for a single coefficient of~$\p_u$ 
is reduced, with no additional assumptions and conditions, to the initial equation~$\mathcal L$ 
by a nonlocal transformation (cf. Corollary~\ref{CorollaryOnLinearizationOfDetEqs0ForRedOpsOfLPEs}).  

The regular case $\tau=1$ is more complicated than the singular one. 
It essentially depends on the structure of the equation including the order, the kind of nonlinearities, etc. 
Up to now there are no exhaustive results on regular reduction operators even for second-order evolution equations. 
Only certain subclasses of such equations were investigated. 
See, e.g., \cite{ArrigoHillBroadbridge1993,Clarkson&Mansfield1993,Popovych2008a,Popovych&Vaneeva&Ivanova2007}
for the complete classifications of regular reduction operators for some subclasses of second-order evolution equations 
parameterized by functions of single arguments. 
For example, even for the class of nonlinear diffusion equations of the general form $u_t=(f(u)u_x)_x$ 
(a classical example of solving a group classification problem for partial differential equations~\cite{Ovsiannikov1982}), 
the set of values of the parameter-function~$f$ which correspond to equations possessing non-Lie regular reduction operators
has not yet been found. 
Most evolution equations have no regular reduction operators. 
A simple example is \[u_t=u_{xx}+ue^{u_x}+xe^{2u_x}+te^{3u_x}+e^{4u_x}+e^{5u_x}.\]
Some evolution equations (the linear ones~\cite{Fushchych&Shtelen&Serov&Popovych1992,Popovych2008a}, Burgers' equation~\cite{Mansfield1999}, etc.) 
possess so many regular reduction operators that `no-go' statements like those for singular reduction operators are true for them,
but the nature of this `no-go' differs from the `no-go' of the singular case and is related to the property of linearity or linearizability 
of the corresponding evolution equations.

\section{Example: nonlinear wave equations}\label{SectionOnExampleOfSimpleNWEs}

The next example which we study in detail within the framework of singular reduction operators 
is given by the class of nonlinear wave equations (in the characteristic, or light-cone, variables) of the  general form 
\begin{equation}\label{EqSimpleNWEs}
u_{12}=F(u).
\end{equation}
Here $F$ is an arbitrary smooth function of~$u$. 
This class essentially differs from the class of evolution equations within the framework of singular vector fields. 
The main differences are that each differential function corresponding to an equation from class~\eqref{EqSimpleNWEs}
has two singular sets of vector fields 
and these sets contain vector fields of lower singularity co-orders than the singularity co-orders of the whole sets.
Thus, for any~$F$ the vector field $Q=\xi^i(x,u)\p_i+\eta(x,u)\p_u$ is singular for the corresponding differential function $L=u_{12}-F(u)$ 
if and only if $\xi^1\xi^2=0$. 
Moreover, it is obvious that there are no differences between strong and weak singularity of vector fields 
for equations from the class~\eqref{EqSimpleNWEs}.
Indeed, suppose that $\xi^2\ne0$. 
Excluding the derivatives $u_2$ and $u_{12}$ from~$L$ according to~\eqref{EqSingularRedOpsExpressionsForU_2}, 
we obtain a differential function $\tilde L$ with the coefficient $-\xi^1/\xi^2$ of $u_{11}$. 
We have $\ord \tilde L<2$ if and only if $\xi^1=0$. 

Therefore, \emph{for any~$F$ the differential function $L=u_{12}-F(u)$ possesses exactly two sets of singular vector fields in the reduced form}, 
$S=\{\p_2+\zeta(x,u)\p_u\}$ and $S^*=\{\p_1+\zeta^*(x,u)\p_u\}$. 
The vector fields equivalent to $\p_u$ are not suitable as reduction operators. 
Any singular vector field of~$L$ is equivalent to one of the above fields.
Moreover, each equation of the form~\eqref{EqSimpleNWEs} admits the discrete symmetry transformation permuting the variables~$x_1$ and~$x_2$. 
This transformation generates a one-to-one mapping between $S$ and~$S^*$ (cf.\ Corollary~\ref{CorollaryOnEquivReductionOperatorWrtSymGroup}). 
Hence it suffices, up to equivalence of vector fields (and permutation of~$x_1$ and~$x_2$), 
to investigate only singular reduction operators from the set~$S$.   

For an equation~$\mathcal L$ from class~\eqref{EqSimpleNWEs} and an operator $Q=\p_2+\zeta\p_u$ the conditional invariance criterion takes the form
\[
(\zeta_{12}+\zeta_{1u}u_2+\zeta_{2u}u_1+\zeta_{uu}u_1u_2+\zeta_uu_{12})|_{\mathcal L\cap\mathcal Q_{(2\textbf{})}}=\zeta F_u.
\]
The intersection $\mathcal L\cap\mathcal Q_{(2)}$ is singled out from~$J^2$ by the equations 
$u_2=\zeta$, $\zeta_1+\zeta_uu_1=F$ and $u_{12}=F$.
Our further considerations therefore depend on the values of $\zeta_u$ and $F_u$. 
We analyze all the possible cases. 

Let $\zeta_u=0$ and $F_u=0$. Then $Q$ is an ultra-singular vector field for the differential function~$L$. 
The third equation defining  $\mathcal L\cap\mathcal Q_{(2)}$ takes the form $\zeta_1=F$ and contains no derivatives of~$u$. 
It should be assumed as a condition with respect to~$\zeta$ 
and hence the conditional invariance criterion is identically satisfied in this case. 
An ansatz constructed with the operator~$Q$ is $u=\varphi(\omega)+\int\!\zeta\, dx_2$, where $\omega=x_1$. 
It reduces equation~\eqref{EqSimpleNWEs} to an identity. 
This is explained by the ultra-singularity of the reduction operator~$Q$. 

If $\zeta_u=0$ and $F_u\ne0$, the singularity co-order of $Q$ for the differential function~$L$ equals 0. 
The third equation defining  $\mathcal L\cap\mathcal Q_{(2)}$ again takes the form $\zeta_1=F$ but now can be solved 
with respect to~$u$: $u=\check F(\zeta_1)$, where $\check F$ is the inverse to~$F$. 
Then the conditional invariance criterion is equivalent to the equation $\zeta_{12}=\zeta F_u(\check F(\zeta_1))$ with respect to $\zeta$.
The ansatz constructed with the operator~$Q$ reduces equation~\eqref{EqSimpleNWEs} 
to the algebraic equation $F(\varphi+\int\!\zeta\, dx_2)=\zeta_1$ for the function $\varphi$ 
in agreement with the singularity co-order $0$ of $Q$. 
Indeed, inverting~$F$, we obtain the equality $\varphi=\check F(\zeta_1)-\int\!\zeta\, dx_2$ 
whose right-hand side does not depend on~$x_2$ in view of the equation on~$\zeta$. 
Conversely, let us fix a solution $u=f(x)$ of equation~\eqref{EqSimpleNWEs} and set $\zeta=f_2$. 
Then $\zeta_{12}=\zeta F_u(\check F(\zeta_1))$, 
i.e., in view of the conditional invariance criterion $Q=\p_2+\zeta\p_u$ is a reduction operator of equation~\eqref{EqSimpleNWEs}, 
and $\zeta_u=0$. The solution $u=f(x)$ is invariant with respect to~$Q$. 
The above results can be summed up as follows:

\begin{theorem}\label{TheoremOnSolutionsAndSingularRedOpsForSimpleWEs}
For any equation from class~\eqref{EqSimpleNWEs} with $F_u\ne0$ 
there exists a one-to-one correspondence between its solutions 
and reduction operators of the form $Q=\p_2+\zeta(x)\p_u$ (resp. $Q^*=\p_1+\zeta^*(x)\p_u$). 
Namely, each operator of this kind is of singularity co-order~$0$ and 
corresponds to the solution which is invariant with respect to this operator. 
The problems of solving an equation from class~\eqref{EqSimpleNWEs} with $F_u\ne0$ 
and the exhaustive description of its reduction operators of the above form are completely equivalent.
\end{theorem}

\begin{corollary}\label{CorollaryOnSolutionsForAdjoint0SingularRedOpsOfSimpleNWEs}
Any solution $u=f(x)$ of equation~\eqref{EqSimpleNWEs} with $F_u\ne0$ is invariant 
with respect to two reduction operators $Q=\p_2+\zeta(x)\p_u$ and $Q^*=\p_1+\zeta^*(x)\p_u$ of equation~\eqref{EqSimpleNWEs}, having singularity co-order~0. 
Here $\zeta=f_2$ and $\zeta^*=f_1$. 
The property of possessing the same invariant solution of equation~\eqref{EqSimpleNWEs} establishes 
a canonical bijection $Q\leftrightarrow Q^*$ between the sets of reduction operators of singularity co-order $0$. 
The adjoint values of $\zeta$ and $\zeta^*$ are connected by the formulas
\[
\zeta^*= \frac{\zeta_{11}}{F_u(\check F(\zeta_1))}, \quad 
\zeta= \frac{\zeta^*_{22}}{F_u(\check F(\zeta^*_2))}. 
\]
\end{corollary}

The regular values of~$\zeta$ for which the singularity co-order of $Q$ coincides with the singularity co-order of the whole family~$S$ 
(and equals 1) satisfy the condition $\zeta_u\ne0$.
The third equation defining  $\mathcal L\cap\mathcal Q_{(2)}$ then provides the following expression for $u_1$: 
\[u_1=\frac{F-\zeta_1}{\zeta_u}=:\zeta^*.\]
The conditional invariance criterion implies only the single equation
\begin{equation}\label{EqDEForSingularRedOpsOfNWEs}
\zeta_{12}+\zeta\zeta_{1u}+(\zeta_{2u}+\zeta\zeta_{uu})\frac{F-\zeta_1}{\zeta_u}+\zeta_uF=\zeta F_u
\end{equation}
with respect to the single function~$\zeta$, 
i.e., in this case the system of determining equations consists of the single equation~\eqref{EqDEForSingularRedOpsOfNWEs} 
and, therefore, is not overdetermined.

Equation~\eqref{EqDEForSingularRedOpsOfNWEs} can be rewritten in the form of the compatibility condition 
\[
\zeta_1+\zeta^*\zeta_u=\zeta^*_2+\zeta\zeta^*_u=F
\]
of the equations $u_1=\zeta^*$, $u_2=\zeta$ and $u_{12}=F$. 
It is obvious that $\zeta^*_u\ne0$. 
Due to symmetry with respect to the permutation of $x_1$ and $x_2$, 
we obtain the following statement. 

\begin{proposition}\label{PropositionOnAdjoint1stCoOrderSingularRedOpsOfSimpleNWEs}
For any equation from class~\eqref{EqSimpleNWEs},  
there exists a canonical bijection $Q\leftrightarrow Q^*$ between sets of its singular reduction operators of the forms 
$Q=\p_2+\zeta(x,u)\p_u$ and $Q^*=\p_1+\zeta^*(x,u)\p_u$, where $\zeta_u\ne0$ and $\zeta^*_u\ne0$. 
This bijection is given by the formulas
\[
Q\to Q^*\colon\quad\zeta^*=\frac{F-\zeta_1}{\zeta_u}, \qquad Q^*\to Q\colon\quad\zeta=\frac{F-\zeta^*_2}{\zeta^*_u}.
\]
A solution of equation~\eqref{EqSimpleNWEs} is invariant with respect to the operator $Q$ 
if and only if it is invariant with respect to the operator $Q^*$.
\end{proposition}

\begin{theorem}\label{TheoremOn1parametricSolutionFamiliesAndSingularRedOpsForSimpleWEs}
Up to the equivalence of solution families, for any equation from class~\eqref{EqSimpleNWEs} with $F_u\ne0$ 
there exists a one-to-one correspondence between one-parametric families of its solutions 
and reduction operators of the form $Q=\p_2+\zeta(x,u)\p_u$, where $\zeta_u\ne0$ (resp. $Q^*=\p_1+\zeta^*(x,u)\p_u$, where $\zeta^*_u\ne0$).
Namely, any such operator corresponds to the family of solutions which are invariant with respect to this operator. 
The problems of the construction of all one-parametric solution families of an equation from class~\eqref{EqSimpleNWEs} with $F_u\ne0$ 
and the exhaustive description of its reduction operators of the above form are completely equivalent.
\end{theorem}

\begin{proof}
In view of Proposition~\ref{PropositionOnAdjoint1stCoOrderSingularRedOpsOfSimpleNWEs}, it is sufficient to consider only 
operators with zero coefficient of $\p_1$. 
Although the proof is similar to the proof of the analogous statement for evolution equations it differs from it 
in essential details and will therefore be presented completely.

An ansatz constructed with the operator $Q=\p_2+\zeta(x,u)\p_u$ has the form $u=f(x,\varphi(\omega))$, 
where $f=f(x,\varphi)$ is a given function, $f_\varphi\ne0$, 
$\varphi=\varphi(\omega)$ is the new unknown function and $\omega=x_1$ is the invariant independent variable.
Here $\zeta_u\ne0$ implies $f_{2\varphi}\ne0$. 
Hence this ansatz reduces equation~\eqref{EqSimpleNWEs} to a first-order ordinary differential equation~$\mathcal L'$ in~$\varphi$, 
which is solvable with respect to~$\varphi'$. 
The general solution of the reduced equation~$\mathcal L'$ essentially depends on an arbitrary constant~$\varkappa$:
$\varphi=\varphi(\omega,\varkappa)$, where $\varphi_\varkappa\ne0$. 
Substituting the general solution into the ansatz gives 
the one-parametric family~$\mathcal F$ of solutions $u=\tilde f(x,\varkappa)$ of~\eqref{EqSimpleNWEs} 
with $\tilde f=f(x,\varphi(x_1,\varkappa))$. 

Conversely, 
suppose that $F_u\ne0$ and $\mathcal F=\{u=f(x,\varkappa)\}$ is a one-parametric family of solutions of~\eqref{EqSimpleNWEs}. 
The derivative $f_\varkappa$ is nonzero since the parameter $\varkappa$ is essential. 
Therefore, $f_{12\varkappa}=f_\varkappa F_u(f)\ne0$.
We express $\varkappa$ from the equality $u=f(x,\varkappa)$: $\varkappa=\Phi(x,u)$ for some function $\Phi=\Phi(x,u)$ 
with $\Phi_u\ne0$. 
Consider the operator $Q=\p_2+\zeta\p_u$, where the coefficient $\zeta=\zeta(x,u)$ is defined by the formula $\zeta=-\Phi_2/\Phi_u$. 
$Q[u]=0$ for any $u\in\mathcal F$. 
The ansatz $u=f(x,\varphi(\omega))$, where $\omega=x_1$, associated with~$Q$, 
reduces~\eqref{EqSimpleNWEs} to the equation $\varphi_\omega=0$ since $f_{2\varkappa}\ne0$.
Therefore \cite{Zhdanov&Tsyfra&Popovych1999}, $Q$ is a reduction operator of equation~\eqref{EqSimpleNWEs} 
and hence the function $\zeta$ satisfies equation~\eqref{EqDEForSingularRedOpsOfNWEs}. 
Moreover, we have $\zeta_u\ne0$ since otherwise the operator $Q$ would reduce~\eqref{EqSimpleNWEs} to an algebraic equation 
with respect to~$\varphi$.
\end{proof}

\begin{corollary}\label{CorollaryOnFamiliesOfInvSolutionsForAdjointSingularRedOpsOfSimpleNWEs}
Any adjoint singular reduction operators $Q=\p_2+\zeta(x,u)\p_u$ and $Q^*=\p_1+\zeta^*(x,u)\p_u$ 
of equation~\eqref{EqSimpleNWEs} (where necessarily $\zeta_u\ne0$ and $\zeta^*_u\ne0$) are associated with 
the same one-parametric family of solutions of this equation. 
\end{corollary}

Let $\zeta$ be an arbitrary solution of equation~\eqref{EqDEForSingularRedOpsOfNWEs}. 
Then $\zeta_u\ne0$ and $Q=\p_2+\zeta(x,u)\p_u$ is a reduction operator of equation~\eqref{EqSimpleNWEs}.  
Consider a one-parametric family $\mathcal F=\{u=f(x,\varkappa)\}$ of solutions of~\eqref{EqSimpleNWEs}, 
which are invariant with respect to~$Q$. 
(Such a family exists in view of Theorem~\ref{TheoremOn1parametricSolutionFamiliesAndSingularRedOpsForSimpleWEs}.)
Expressing the parameter~$\varkappa$ from the equality  $u=\tilde f(x,\varkappa)$, 
we obtain that $\varkappa=\Phi(x,u)$, where $\Phi_u\ne0$. 
$\zeta=u_2=-\Phi_2/\Phi_u$ for any $u\in\mathcal F$, i.e., for any admissible values of $(x,\varkappa)$. 
This implies that the representation $\zeta=-\Phi_2/\Phi_u$ is true for any admissible value of $(x,u)$. 
This provides the background for the following statement.

\begin{corollary}\label{CorollaryOnReductionOfDetEqsForSingularRedOpsOfSimpleWEs}
The nonlinear three-dimensional equation~\eqref{EqDEForSingularRedOpsOfNWEs} is reduced by composition of 
the B\"acklund transformation $\zeta=-\Phi_2/\Phi_u$, $\zeta^*=-\Phi_1/\Phi_u$,
where $\Phi$ is a function of $(x,u)$, and the hodograph transformation 
\begin{gather*}
\mbox{the new independent variables:}\qquad\tilde x_1=x_1, \quad \tilde x_2=x_2, \quad \varkappa=\Phi, 
\\  
\lefteqn{\mbox{the new dependent variable:}}\phantom{\mbox{the new independent variables:}\qquad }\tilde u=u 
\end{gather*}
to the equation~\eqref{EqSimpleNWEs} for the function $\tilde u=\tilde u(\tilde x,\varkappa)$ 
with $\varkappa$ playing the role of a parameter.
\end{corollary}

\begin{proof}
We take an arbitrary solution $\zeta$ of equation~\eqref{EqDEForSingularRedOpsOfNWEs} 
(the condition $\zeta_u\ne0$ is implicitly assumed to be satisfied) and set $\zeta^*=(F-\zeta_1)/\zeta_u$. 
In view of the Frobenius theorem, the system $\Phi_2+\zeta\Phi_u=0$, $\Phi_1+\zeta^*\Phi_u=0$ 
with respect to the function $\Phi=\Phi(x,u)$ is compatible since 
its compatibility condition $\zeta_1+\zeta^*\zeta_u=\zeta^*_2+\zeta\zeta^*_u$ coincides with~\eqref{EqDEForSingularRedOpsOfNWEs} 
and hence is identically satisfied. 
We choose a nonconstant solution~$\Phi$ of this system. 
Then $\Phi_u\ne0$, $\zeta=-\Phi_2/\Phi_u$ and $\zeta^*=-\Phi_1/\Phi_u$. 
After the hodograph transformation, 
the latter equations take the form $\tilde u_{\tilde x_2}=\zeta(\tilde x,\tilde u)$ and $\tilde u_{\tilde x_1}=\zeta^*(\tilde x,\tilde u)$. 
This directly implies that for any value of $\varkappa$ the function $\tilde u=\tilde u(\tilde x,\varkappa)$ satisfies equation~\eqref{EqSimpleNWEs}. 
The parameter $\varkappa$ is essential in $\tilde u$ since $\tilde u_\varkappa=1/\Phi_u\ne0$.

It follows from the proof of Theorem~\ref{TheoremOn1parametricSolutionFamiliesAndSingularRedOpsForSimpleWEs} 
that the application of the inverse transformations to a one-parametric family of solutions of equation~\eqref{EqSimpleNWEs} 
results in a solution of equation~\eqref{EqDEForSingularRedOpsOfNWEs}.
\end{proof}

\begin{note}\label{NoteOn1parametricSolutionFamiliesAndSingularRedOpsForSimpleWEsWithFu0}
For any equation from class~\eqref{EqSimpleNWEs} with $F_u=0$,  
reduction operators of the form $Q=\p_2+\zeta(x,u)\p_u$, where $\zeta_u\ne0$ (resp. $Q^*=\p_1+\zeta^*(x,u)\p_u$, where $\zeta^*_u\ne0$) 
also are bijectively associated with one-parametric families of its solutions, 
having the form $\{u=f(x,\varkappa)\}$ where $f_{1\varkappa}\ne0$ (resp. $f_{2\varkappa}\ne0$). 
The one-parametric families with $f_{1\varkappa}=0$ (resp. $f_{2\varkappa}=0$) necessarily existing in this case
correspond to ultra-singular reduction operators with $\zeta_u=0$ (resp. $\zeta^*_u=0$), and the correspondence is not one-to-one. 
\end{note}

The above investigation of singular reduction operators of nonlinear wave equations of the form~\eqref{EqSimpleNWEs}  
shows that for these equations the natural partition of the corresponding sets of reduction operators is 
into triples of subsets singled out by the conditions 
\[
1)\ \xi^1=0; \quad 2)\ \xi^2=0;  \quad 3)\ \xi^1\xi^2\ne0.
\]
After the factorization with respect to the equivalence relation of vector fields, 
we obtain three subsets of reduction operators, which have to be investigated separately.
The defining conditions for these subsets are, respectively, 
\[
1)\ \xi^1=0,\ \xi^2=1;\quad 2)\ \xi^2=0,\ \xi^1=1;  \quad 3)\ \xi^1\ne0,\ \xi^2=1.
\]
Since any equation from class~\eqref{EqSimpleNWEs} admits the point symmetry permuting~$x_1$ and~$x_2$, 
the second case is reduced to the first one and can be omitted. 
Finally we have two essentially different cases after factorization: 
the singular case $\xi^1=0$, $\xi^2=1$ and the regular case $\xi^1\ne0$, $\xi^2=1$.
The gauge $\xi^2=1$ is not uniquely possible in the regular case and may be varied for optimizing
the further consideration of this case. 

Consider the other standard form  
\begin{equation}\label{EqSimpleNWEsInUsualVariables}
u_{11}-u_{22}=F(u)
\end{equation}
of nonlinear wave equations, obtained from~\eqref{EqSimpleNWEs} via the point transformation 
$\tilde x_1=x_1-x_2$, $\tilde x_2=x_1+x_2$, $\tilde u=u$. 
Using this transformation, all the results derived for class~\eqref{EqSimpleNWEs} 
can easily be extended to class~\eqref{EqSimpleNWEsInUsualVariables}. 
Thus, any equation of the form~\eqref{EqSimpleNWEsInUsualVariables} 
possesses two singular sets of reduction operators, singled out by the conditions $\xi^1=-\xi^2$ and $\xi^1=\xi^2$,
and one regular set of reduction operators, associated with the condition $\xi^1\ne\pm\xi^2$. 
The singular sets are mapped to each other by alternating the sign of~$x_2$ 
and hence one of them can be excluded from the consideration.  
After factorization with respect to the equivalence relation of vector fields, 
we have two cases for our further study:
the singular case $\xi^1=\xi^2=1$ and the regular case $\xi^1\ne\pm1$, $\xi^2=1$.

For nonlinear wave equations of the general form \[u_{11}-(G(u)u_2)_2=F(u),\] where $G(u)>0$, 
the natural partitions of the sets of reduction operators are determined by more complicated conditions 
depending on the parameter-function~$G$. 
We will not discuss these equation here. 
We only remark that the singular sets of the corresponding reduction operators are associated with the conditions 
$\xi^2=\sqrt{G}\xi^1$ and $\xi^2=-\sqrt{G}\xi^1$, respectively.

The above examples underline that the application of the conventional partition for factorization of sets of reduction operators 
often leads to the splitting of uniform cases and to combining essentially different ones. 
As a result, the derived systems of determining equations for the coefficients of reduction operators 
is far from optimal and difficult to investigate. 
Therefore, natural partitions based on taking into account the structure of singular families 
of reduction operators offers a decisive advantage.

\section{Reduction operators and parametric families of solutions}%
\label{SectionOnReductionOperatorsAndParametricFamiliesOfSolutions}

\begin{proposition}\label{PropositionOnSingularRedOpsAndReducedEqs2DCase}
Let $Q$ be a reduction operator of an equation~$\mathcal L$. 
Then the weak singularity co-order of~$Q$ for~$\mathcal L$ equals
the essential order of the corresponding reduced ordinary differential equation.
\end{proposition}

\begin{proof}
We carry out a point transformation in such a way that in the new variables the operator~$Q$ has the form $Q=\p_{x_2}$. 
(For convenience, for the new variables we use the same notations as for the old ones.)
Then an ansatz constructed with~$Q$ is $u=\varphi(\omega)$, 
where $\varphi=\varphi(\omega)$ is the new unknown function and $\omega=x_1$ is the invariant independent variable. 
The manifold~$\mathcal Q_{(r)}$ is defined by the system $u_\alpha=0$, 
where $\alpha=(\alpha_1,\alpha_2)$, $\alpha_2>0$, $\alpha_1+\alpha_2\leqslant r=\ord L$.

Since $Q\in\mathcal Q(\mathcal L)$, there exist differential functions 
$\smash{\check\lambda=\check\lambda[\varphi]}$ and $\smash{\check L=\check L[\varphi]}$ 
of an order not greater than~$r$ such that $L|_{u=\varphi(\omega)}=\check\lambda\check L$
(cf.\ \cite{Zhdanov&Tsyfra&Popovych1999}). 
The function~$\check\lambda$ does not vanish and may depend on~$x_2$ as a parameter. 
The function~$\check L$ is assumed to be of minimal order~$\check r$ 
which may be attained up to the equivalence generated by nonvanishing multipliers. 
Then the reduced equation~$\check{\mathcal L}$: $\check L=0$ has essential order~$\check r$. 

The condition $\wsco_{\mathcal L}Q=k$ means that
there exists a strictly $k$th order differential function $\tilde L=\tilde L[u]$ 
and a nonvanishing differential function $\tilde\lambda=\tilde\lambda[u]$ of an order not greater than~$r$, 
which depend at most on~$x$ and derivatives of~$u$ with respect to~$x_1$,
such that $L|_{\mathcal Q_{(r)}}=\tilde\lambda\tilde L|_{\mathcal Q_{(r)}}$. 

If $\check r$ would be less than~$k$, 
we could use $\tilde\lambda_{\rm new}=\check\lambda|_{u\rightsquigarrow\varphi}$ and $\tilde L_{\rm new}=\check L|_{u\rightsquigarrow\varphi}$ 
in the definition of weak singularity and would arrive at the contradiction $\wsco_{\mathcal L}Q\leqslant\ord\tilde L_{\rm new}=\check r<k$. 
Therefore, $\check r\geqslant k$.
(Here, ``$y\rightsquigarrow z$'' means that the value $y$ should be substituted instead of the value~$z$.)

Suppose that $\check r>k$. We have the equality $\check\lambda\check L=(\tilde\lambda\tilde L)|_{u=\varphi(\omega)}$ 
in which the variable~$x_2$ plays the role of a parameter. 
Fixing a value $x_2^0$ of $x_2$, we obtain the representation
\[
\check L=\Lambda[\varphi]\,\tilde L\biggr|_{u=\varphi(\omega),\; x_2=x_2^0}, \quad 
\Lambda:=\frac{\tilde\lambda|_{u=\varphi(\omega)}}{\check\lambda}\biggr|_{x_2=x_2^0}\ne0. 
\]
Since $\ord\tilde L|_{u=\varphi(\omega),\; x_2=x_2^0}\leqslant k< \check r$, this representation contradicts 
the condition that $\check r$ is the essential order of the reduced equation~$\check{\mathcal L}$.
Therefore, $\check r=k$.
The inverse change of variables preserves the claimed property. 
\end{proof}

\begin{corollary}\label{CorollaryOnSingularRedOpsAndFamiliesOfInvSolutions2DCaseA}
Let $Q$ be a reduction operator of an equation~$\mathcal L$. 
Then the weak singularity co-order of~$Q$ for~$\mathcal L$ equals 
the maximal number of essential parameters in families of $Q$-invariant solutions of~$\mathcal L$.
\end{corollary}

\begin{proof}
The essential order $\check r$ of the reduced ordinary differential equation~$\check{\mathcal L}$ associated with~$Q$ 
coincides with the weak singularity co-order of~$Q$ for~$\mathcal L$. 
The maximal number of essential parameters in solutions of~$\check{\mathcal L}$ equals the order of~$\check{\mathcal L}$. 
The substitution of these solutions into the corresponding ansatz leads to parametric families of $Q$-invariant solutions of~$\mathcal L$, 
and all $Q$-invariant solutions of~$\mathcal L$ are obtained in this way. 
Therefore, the maximal number of essential parameters in families of $Q$-invariant solutions of~$\mathcal L$ equals $\check r$.
\end{proof}

\begin{corollary}\label{CorollaryOnSingularRedOpsAndFamiliesOfInvSolutions2DCaseB}
Let $Q$ be a $k$th co-order weakly singular reduction operator of an equation~$\mathcal L$. 
Suppose additionally that a differential function of minimal order, 
associated with~$L$ on the manifold~$\mathcal Q_{(r)}$ up to a nonvanishing multiplier, 
is of maximal rank in the derivative of~$u$ of the highest order~$k$ appearing in this differential function. 
Then $\mathcal L$ possesses a $k$-parametric family of $Q$-invariant solutions, 
and any $Q$-invariant solution of~$\mathcal L$ belongs to this family. 
\end{corollary}

\begin{proof}
Under this assumption, the reduced ordinary differential equation~$\check{\mathcal L}$ associated with~$Q$ can be written in normal form and hence 
has a $k$-parametric general solution which contains all solutions of~$\check{\mathcal L}$.   
Substituting it into the corresponding ansatz, this solution gives a $k$-parametric family of $Q$-invariant solutions of~$\mathcal L$. 
There are no other $Q$-invariant solutions of~$\mathcal L$.
\end{proof}

\begin{corollary}\label{CorollaryOnSingularRedOpsAndFamiliesOfInvSolutions2DCaseC}
Suppose that a differential function of minimal order, 
associated with~$L$ on the manifold~$\mathcal Q_{(r)}$ up to a nonvanishing multiplier, 
is of maximal rank in the highest order derivative of~$u$ appearing in this differential function. 
If the maximal number of essential parameters in families of $Q$-invariant solutions of~$\mathcal L$ 
is not less than the weak singularity co-order of~$Q$ for~$\mathcal L$ 
then $Q$ is a reduction operator of~$\mathcal L$. 
\end{corollary}

\begin{proof}
Point transformations of the variables do not change the claimed property. 
We use the variables and notations from the proof of Proposition~\ref{PropositionOnSingularRedOpsAndReducedEqs2DCase}. 
Consider the differential function $\hat L[\varphi]=\tilde L|_{u=\varphi(\omega)}$. 
It depends on~$x_2$ as a parameter and $\ord \hat L=k$. 
Due to the condition of maximal rank, we can resolve the equation $\hat L=0$ with respect to the highest order derivative $\varphi^{(k)}$: 
$\varphi^{(k)}=R[\varphi]$, where  $\ord R<k$. 

If $R_{x_2}\ne0$, splitting with respect to~$x_2$ in the equation $\hat L=0$ results in 
an ordinary differential equation $\tilde R[\varphi]=0$ of an order lower than~$k$. 
Any $Q$-invariant solution of~$\mathcal L$ has the form $u=\varphi(\omega)$, where the function~$\varphi$ satisfies, in particular, 
the equation $\tilde R[\varphi]=0$. 
This contradicts the condition that 
the maximal number of essential parameters in families of $Q$-invariant solutions of~$\mathcal L$ is not less than~$k$. 

Therefore,  $R_{x_2}=0$, i.e., 
the equation $\varphi^{(k)}=R[\varphi]$ is a reduced equation which is obtained from~$\mathcal L$ by the substitution 
of the ansatz $u=\varphi(\omega)$ constructed with the operator~$Q=\p_2$.
\end{proof}

\begin{note}\label{NoteOnParameterNumberAndWeakSingularityCoOrder}
For any operator~$Q$, 
the maximal number of essential parameters in families of $Q$-invariant solutions of~$\mathcal L$ 
cannot be greater than $\wsco_L Q$.
\end{note}

Summing up the above consideration, we can formulate the following statement. 

\begin{proposition}\label{PropositionOnSingularRedOpsAndFamiliesOfInvSolutions2DCase}
Suppose that a differential function of minimal order, 
associated with the differential function~$L[u]$ on the manifold~$\mathcal Q_{(r)}$ ($r=\ord L$) up to a nonvanishing multiplier, 
is of maximal rank in the highest order derivative of~$u$ appearing in this differential function. 
Then any two of the following properties imply the third one.

1)  $Q$ is a reduction operator of the equation~$\mathcal L$: $L=0$.

2) The weak singularity co-order of~$Q$ for~$\mathcal L$ equals~$k$ ($0\leqslant k\leqslant r$).

3) The equation~$\mathcal L$ possesses a $k$-parametric family of $Q$-invariant solutions, 
and any $Q$-invariant solution of~$\mathcal L$ belongs to this family.
\end{proposition}

The properties of ultra-singular vector fields as reduction operators are obvious. 

\begin{proposition}\label{PropositionOnUltraSingularRedOpsAndFamiliesOfInvSolutions2DCase}
1) Any ultra-singular vector field~$Q$ of a differential equation~$\mathcal L$ is a reduction operator of this equation. 
An ansatz constructed with~$Q$ reduces~$\mathcal L$ to the identity. 
Therefore, the family of $Q$-invariant solutions of~$\mathcal L$ is parameterized by an arbitrary function 
of a single $Q$-invariant variable.

2) If the family of $Q$-invariant solutions of~$\mathcal L$ is parameterized by an arbitrary function 
of a single $Q$-invariant variable then $Q$ is an ultra-singular vector field for~$\mathcal L$.
\end{proposition}

\section{Reduction operators of singularity co-order~1}\label{SectionOnReductionOperatorsOfSingularityCoOrder1}

Encouraged by the above investigation of evolution and, especially, wave equations, 
we study co-order one singular reduction operators of general partial differential equations in 
two independent and one dependent variables. 

Consider an equation~$\mathcal L$: $L=0$, where $L=L[u]$ is a differential function of order $r>1$.
Suppose that the function~$L$ admits a first co-order singular module of vector fields. 
(In view of Corollary~\ref{CorollaryUselessnessOfWeaklySingularFamiliesOfVectorFields}, 
we can restrict ourselves to considering only strong singularity of vector fields for differential equations.)
Without loss of generality, up to changing variables we can assume that the module contains 
a first co-order singular set $S=\{Q^\zeta\}$ of vector fields in reduced form, i.e., 
$Q^\zeta=\xi\p_1+\p_2+\zeta\p_u$ for any smooth function~$\zeta$ of $(x,u)$ and a fixed smooth function~$\xi$.
Additionally, we can assume $\xi\in\{0,u\}$.

By Theorem~\ref{TheoremOn2DSingularVectorFieldsOfCoOrderK}, 
the differential function~$L$ can be written in the form $L=\check L(x,\Omega_{r,1})$,
where 
\[\Omega_{r,1}=\bigl(\omega_\alpha=D_1^{\alpha_1}(\xi D_1+D_2)^{\alpha_2}u,\alpha_1\leqslant 1,\alpha_1+\alpha_2\leqslant r\bigr),\] 
and $\check L_{\omega_\alpha}\ne0$ for some $\omega_\alpha$ with $\alpha_1=1$.
Then the restriction of~$L$ to $\smash{\mathcal Q^\zeta_{(r)}}$ coincides with 
the restriction, to the same manifold $\smash{\mathcal Q^\zeta_{(r)}}$, 
of the function $\tilde L^\zeta=\check L(x,\tilde\Omega_{r,1})$, where 
\[\tilde\Omega_{r,1}=\bigl(D_1^{\alpha_1}(Q^\zeta)^{\alpha_2}u,\alpha_1\leqslant 1,\alpha_1+\alpha_2\leqslant r\bigr).\] 
Thus, the form of $\tilde L^\zeta$ is determined by the forms of $L$ and $\xi$ and a chosen value of the parameter-function~$\zeta$.
Depending on the value of~$\zeta$, the differential function~$\tilde L^\zeta$ may either  identically vanish or be of order~0 or~1. 
This means that either the vector field $Q^\zeta$ is ultra-singular or $\sco_L Q^\zeta=0$ or $\sco_L Q^\zeta=1$, respectively. 
We investigate each of the above cases separately. 
Below we additionally suppose that the function $\tilde L^\zeta$ is of maximal rank with respect to $u$ (resp. $u_1$) 
if $\sco_L Q^\zeta=0$ (resp. $\sco_L Q^\zeta=1$). 

The values of~$\zeta$ for which $Q^\zeta$ for~$\mathcal L$ is ultra-singular are singled out by the condition $\tilde L^\zeta=0$, 
where $u$ and $u_1$ are considered as independent variables. 
Splitting this condition with respect to~$u_1$ gives a system $\mathcal S_{-1}$ of partial differential equations in~$\zeta$ 
of orders less than~$r$, which may be incompatible in the general case. 
The incompatibility of this system means that the set~$S$ contains no ultra-singular vector fields. 
For example, evolution equations of orders greater than 1 
and nonlinear wave equations of the form~\eqref{EqSimpleNWEs} with $F_u\ne0$,  
in contrast to equations of the form~\eqref{EqSimpleNWEs} with $F_u=0$, have no ultra-singular vector fields, 
see~Sections~\ref{SectionOnExampleOfEvolEqs} and~\ref{SectionOnExampleOfSimpleNWEs}. 
$\zeta$ satisfying the ultra-singularity condition guarantees that $Q^\zeta\in\mathfrak Q(\mathcal L)$ 
and the family of $Q^\zeta$-invariant solutions of~$\mathcal L$ is parameterized by an arbitrary function of a single $Q^\zeta$-invariant variable.

If $\sco_L Q^\zeta=0$, the parameter-function~$\zeta$ satisfies the condition $\tilde L^\zeta_{u_1}=0$ 
with $u$ and $u_1$ viewed as independent variables, which is weaker than the ultra-singularity condition. 
Therefore, the corresponding system~$\mathcal S_0$ of partial differential equations in~$\zeta$ of orders less than~$r$, 
obtained by splitting the zero co-order singularity condition with respect to~$u_1$, 
has more chances of being compatible than~$\mathcal S_{-1}$. 
Thus, any nonlinear wave equation of the form~\eqref{EqSimpleNWEs} with $F_u\ne0$ admits 
zeroth co-order singular vector fields although this is not the case for ultra-singular vector fields.
At the same time, evolution equations do not possess zeroth co-order singular vector fields. 

Certain conditions which are sufficient for the compatibility of~$\mathcal S_0$ can be formulated.
Thus, if $\check L_{\omega_{(1,0)}}=0$ and $\xi_u=0$ then the system $\mathcal S_0$ is compatible 
since it is satisfied by any~$\zeta$ with $\zeta_u=0$. 
In other words, $\sco_L Q^\zeta\leqslant0$ for any $\zeta=\zeta(x)$. 
Let us consider this particular case in more detail. 
(Recall that under the condition $\xi_u=0$ the coefficient~$\xi$ can be assumed, up to point transformations, to equal~0 
but we will not use this possibility.)

If additionally $\check L_{\omega_{(0,0)}}=0$, 
the condition $\tilde L^\zeta=0$ under the assumption $\zeta=\zeta(x)$
implies only a single partial differential equation with respect to~$\zeta$. 
Any of its solutions is a solution of $\mathcal S_{-1}$ and hence 
the corresponding vector field $Q^\zeta$ is ultra-singular for~$L$. 

Otherwise $\sco_L Q^\zeta=0$ and we can resolve the equation $\tilde L^\zeta=0$ with respect to~$u$: 
$u=G^\zeta(x)$, where the expression for the function $G^\zeta$ depends on the parameter-function~$\zeta=\zeta(x)$ 
and its derivatives up to order~$r-1$. 
Then the conditional invariance criterion is equivalent to 
the $r$th order partial differential equation $\zeta=\xi G^\zeta_1+ G^\zeta_2$ with respect to $\zeta$.
If $\zeta$ is a solution of this equation then $Q^\zeta$ is a reduction operator of~$\mathcal L$. 
The ansatz constructed with the operator $Q^\zeta$ can be taken in the form $u=\varphi(\omega)+G^\zeta(x)$, 
where $\varphi=\varphi(\omega)$ is the new unknown function and 
$\omega=\omega(x)$ is the invariant independent variable satisfying the equation $\xi\omega_1+\omega_2=0$. 
It reduces the initial equation~$\mathcal L$ to a trivial algebraic equation $\varphi=0$, 
i.e., the function $u=G^\zeta(x)$ is a unique $Q^\zeta$-invariant solution of~$\mathcal L$.
Conversely, let us fix a solution $u=f(x)$ of the equation~$\mathcal L$ and set $\zeta=\xi f_1+ f_2$. 
Then $f=G^\zeta(x)$ and hence $\zeta=\xi G^\zeta_1+ G^\zeta_2$, 
i.e., in view of the conditional invariance criterion $Q^\zeta=\xi\p_1+\p_2+\zeta\p_u$ is a reduction operator of~$\mathcal L$, 
and $\zeta_u=0$. The solution $u=f(x)$ is invariant with respect to~$Q^\zeta$ by construction. 
Thus we obtain:

\begin{theorem}\label{TheoremOnSolutionsAndZeroOrderSingularRedOpsForGen2DPDEs}
Suppose that an equation~$\mathcal L$: $L=0$ possesses a first co-order singular set 
$S=\{Q^\zeta\}$ of vector fields in reduced form $Q^\zeta=\xi\p_1+\p_2+\zeta\p_u$ with $\xi_u=0$, i.e., 
its right hand side~$L$ is represented in the form $L=\check L(x,\Omega_{r,1})$,
where \[\Omega_{r,1}=\bigl(\omega_\alpha=D_1^{\alpha_1}(\xi D_1+D_2)^{\alpha_2}u,\alpha_1\leqslant 1,\alpha_1+\alpha_2\leqslant r\bigr),\] 
$\check L_{\omega_\alpha}\ne0$ for some $\alpha$ with $\alpha_1=1$, 
and additionally $\check L_{\omega_{(1,0)}}=0$ and $\check L_{\omega_{(0,0)}}\ne0$. 
Then there exists a one-to-one correspondence between solutions of~$\mathcal L$  
and reduction operators from~$S$ with $\zeta_u=0$.
Namely, any such operator is of singularity co-order~$0$ and 
corresponds to the unique solution which is invariant with respect to this operator. 
The problems of solving the equation~$\mathcal L$
and the exhaustive description of its reduction operators of the above form are completely equivalent.
\end{theorem}

Now we consider the regular values of~$\zeta$ for which the singularity co-order of $Q^\zeta$ 
coincides with the singularity co-order of the whole family~$S$ (and equals 1). 
If $\sco_L Q^\zeta=1$, the parameter-function~$\zeta$ satisfies the regularity condition 
$\tilde L^\zeta_{u_1}\ne0$. 
Therefore, the equation $\tilde L^\zeta=0$ 
which is equivalent to~$\mathcal L$ on the manifold~$\smash{\mathcal Q^\zeta_{(r)}}$ 
can be solved with respect to~$u_1$: $u_1=G^\zeta(x,u)$, 
where the expression for the function $G^\zeta$ depends on the parameter-function~$\zeta$ 
and its derivatives up to order~$r-1$.
Applied to the equation~$\mathcal L$ and the operator~$Q^\zeta$, 
the conditional invariance criterion implies only the equation
\begin{equation}\label{EqDEFor1stOrderSingularRedOpsOfGen2DPDEs}
\zeta_1+\zeta_uG^\zeta-(\xi_1+\xi_uG^\zeta)G^\zeta=\xi G^\zeta_1+G^\zeta_2+\zeta G^\zeta_u
\end{equation}
with respect to the function~$\zeta$.  
Therefore, in this case the system of determining equations consists of 
the single equation~\eqref{EqDEFor1stOrderSingularRedOpsOfGen2DPDEs} and, therefore, is not overdetermined. 
This equation can be rewritten as the compatibility condition 
\[
\zeta_1+\zeta_uG^\zeta-(\xi_1+\xi_uG^\zeta)G^\zeta-\xi(G^\zeta_1+G^\zeta_uG^\zeta)=G^\zeta_2+(\zeta-\xi G^\zeta)G^\zeta_u
\]
of the equations $u_1=G^\zeta$ and $\xi u_1 + u_2=\zeta$ with respect to~$u$. 
The order of~\eqref{EqDEFor1stOrderSingularRedOpsOfGen2DPDEs} equals $r$ 
and hence is greater than the order of the system~$\mathcal S_0$. 
This guarantees (under certain conditions of smoothness, e.g., in the analytical case) 
that the equation~\eqref{EqDEFor1stOrderSingularRedOpsOfGen2DPDEs} has solutions which are not solutions of~$\mathcal S_0$. 
In other words, \emph{the equation~$\mathcal L$ necessarily possesses first co-order singular reduction operators 
which belong to~$S$.}

The results of Section~\ref{SectionOnReductionOperatorsAndParametricFamiliesOfSolutions} imply
that for each first co-order singular reduction operator~$Q$ of the equation~$\mathcal L$  
there exist a one-parametric family of $Q$-invariant solutions of~$\mathcal L$. 
If the equation~$\mathcal L$ admits a co-order one singular module of vector fields, 
the converse statement is true as well.   

\begin{theorem}\label{TheoremOn1parametricSolutionFamiliesAndSingularRedOpsForGen2DPDEs}
Suppose that an equation~$\mathcal L$: $L=0$ possesses a co-order one singular set 
$S=\{Q^\zeta\}$ of vector fields in reduced form $Q^\zeta=\xi\p_1+\p_2+\zeta\p_u$.
Then for any one-parametric family~$\mathcal F$ of solutions of~$\mathcal L$ 
there exist a value of the parameter-function $\zeta=\zeta(x,u)$ such that 
$Q^\zeta$ is a reduction operator of~$\mathcal L$
and each solution from~$\mathcal F$ is invariant with respect to~$Q^\zeta$. 
\end{theorem}

\begin{proof}
Consider a one-parametric family $\mathcal F=\{u=f(x,\varkappa)\}$ of solutions of~$\mathcal L$. 
The derivative $f_\varkappa$ is nonzero since the parameter $\varkappa$ is essential. 
From $u=f(x,\varkappa)$ we derive $\varkappa=\Phi(x,u)$ 
with some function $\Phi=\Phi(x,u)$, where $\Phi_u\ne0$, 
and then define $\zeta=\zeta(x,u)$ by the formula 
\[\zeta=-\frac{\xi\Phi_1+\Phi_2}{\Phi_u}.\]
Since $f_i=-(\Phi_i/\Phi_u)|_{u=f}$, $i=1,2$, then $\xi|_{u=f} f_1+f_2=\zeta|_{u=f}$, 
i.e., any solution from  $\mathcal F$ is $Q^\zeta$-invariant. 
Then either $Q^\zeta$ is an ultra-singular vector field for~$L$ or $\sco_L Q^\zeta=1$. 
(The case $\sco_L Q^\zeta=0$ is impossible since otherwise the equation~$\mathcal L$ 
could not have a one-parametric family of $Q^\zeta$-invariant solutions.)
Any ultra-singular vector field for~$L$ is a reduction operator of~$\mathcal L$.  
If $\sco_L Q^\zeta=1$ then $Q$ is a reduction operator of~$\mathcal L$ 
in view of Corollary~\ref{CorollaryOnSingularRedOpsAndFamiliesOfInvSolutions2DCaseC}.
\end{proof}

\begin{corollary}\label{CorollaryOnOn1parametricSolutionFamiliesAndFirstCoOrderSingularRedOpsForGen2DPDEs}
Suppose that an equation~$\mathcal L$ possesses a first co-order singular set 
$S=\{Q^\zeta\}$ of vector fields in reduced form $Q^\zeta=\xi\p_1+\p_2+\zeta\p_u$, 
and that no element of~$S$ is ultra-singular for~$\mathcal L$.
Then up to the equivalence of solution families 
there exists a bijection between one-parametric families of solutions of~$\mathcal L$
and its first co-order singular reduction operators belonging to~$S$.
Namely, each operator of this kind corresponds to the family of solutions which are invariant under it. 
The problems of the construction of all one-parametric solution families of the equation~$\mathcal L$ 
and the exhaustive description of its reduction operators of the above form are completely equivalent.
\end{corollary}

This bijection is broken in the presence of ultra-singular vector fields.

The above relation between one-parametric families of solutions and first co-order singular reduction operators 
can be stated as a connection between the initial equation~$\mathcal L$ 
and the determining equation~\eqref{EqDEFor1stOrderSingularRedOpsOfGen2DPDEs}.  

\begin{corollary}\label{CorollaryOnReductionOfDetEqsForFirstCoOrderSingularRedOpsForGen2DPDEs}
Suppose that an equation~$\mathcal L$: $L=0$ possesses a first co-order singular set 
$S=\{Q^\zeta\}$ of vector fields in reduced form $Q^\zeta=\xi\p_1+\p_2+\zeta\p_u$.
Then the determining equation for values of $\zeta$ 
corresponding to first co-order singular reduction operators of~$\mathcal L$ 
is reduced by composition of 
the B\"acklund transformation 
\[\xi\Phi_1+\Phi_2+\zeta\Phi_u=0,\quad \Phi_1+G^\zeta\Phi_u=0\]
where $\Phi$ is a function of $(x,u)$, and the hodograph transformation 
\begin{gather*}
\mbox{the new independent variables:}\qquad\tilde x_1=x_1, \quad \tilde x_2=x_2, \quad \varkappa=\Phi, 
\\  
\lefteqn{\mbox{the new dependent variable:}}\phantom{\mbox{the new independent variables:}\qquad }\tilde u=u 
\end{gather*}
to the initial equation~$\mathcal L$ for the function $\tilde u=\tilde u(\tilde x,\varkappa)$ 
with $\varkappa$ playing the role of a parameter.
\end{corollary}

\begin{proof}
We fix an arbitrary solution $\zeta$ of equation~\eqref{EqDEFor1stOrderSingularRedOpsOfGen2DPDEs}, 
which additionally satisfies the condition $\tilde L^\zeta_{u_1}\ne0$. 
In view of the Frobenius theorem, the equations $\xi\Phi_1+\Phi_2+\zeta\Phi_u=0$ and $\Phi_1+G^\zeta\Phi_u=0$
are compatible with respect to the function $\Phi=\Phi(x,u)$ since 
their compatibility condition coincides with~\eqref{EqDEFor1stOrderSingularRedOpsOfGen2DPDEs}
and hence is identically satisfied. 
We choose a nonconstant solution~$\Phi$ of both these equations. 
Then $\Phi_u\ne0$ and 
\[
\zeta=-\xi\frac{\Phi_1}{\Phi_u}+\frac{\Phi_2}{\Phi_u}, \quad G^\zeta=-\frac{\Phi_1}{\Phi_u}.
\] 
After the hodograph transformation, 
the latter equations take the form 
$\xi\tilde u_{\tilde x_1}+\tilde u_{\tilde x_2}=\zeta(\tilde x,\tilde u)$ and 
$\tilde u_{\tilde x_1}=G^\zeta(\tilde x,\tilde u)$. 
This directly implies that for any value of $\varkappa$ the function $\tilde u=\tilde u(\tilde x,\varkappa)$ satisfies 
the equation~$\mathcal L$. 
The parameter $\varkappa$ is essential in $\tilde u$ since $\tilde u_\varkappa=1/\Phi_u\ne0$.

It follows from the proof of Theorem~\ref{TheoremOn1parametricSolutionFamiliesAndSingularRedOpsForGen2DPDEs} 
that the application of the inverse transformations to a one-parametric family of solutions of the initial equation~$\mathcal L$ 
results in a solution of equation~\eqref{EqDEFor1stOrderSingularRedOpsOfGen2DPDEs}  
if the defined value of~$\zeta$ satisfies the regularity condition $\tilde L^\zeta_{u_1}\ne0$.
\end{proof}

\subsection*{Acknowledgements}

The authors are grateful to Vyacheslav Boyko for useful discussions and interesting comments.
MK was supported by START-project Y237 of the Austrian Science Fund. 
The research of ROP was supported by the Austrian Science Fund (FWF), project P20632.


\begin{thebibliography}{99}\itemsep=-.07ex
\footnotesize

\bibitem{ArrigoHillBroadbridge1993}
Arrigo D.J., Hill J.M., Broadbridge P.,
Nonclassical symmetry reductions of the linear diffusion equation with a nonlinear source,
{\it IMA J. Appl. Math.}, 1994, {\bf 52}, 1--24.

\bibitem{Bila&Niesen2004}
Bila N. and Niesen J., On a new procedure for finding nonclassical symmetries, 
{\it J. Symbolic Comput.}, 2004, {\bf 38}, 1523--1533. 

\bibitem{Bluman&Cole1969}
Bluman~G.W. and Cole~J.D.,
The general similarity solution of the heat equation,
{\it J. Math. Mech.}, 1969, {\bf 18}, 1025--1042.

\bibitem{Clarkson&Kruskal1989}
Clarkson P.A. and Kruskal M.D., 
New similarity solutions of the Boussinesq equation, 
{\it J. Math. Phys.}, 1989, {\bf 30}, 2201--2213.

\bibitem{Clarkson&Mansfield1993}
Clarkson P.A. and Mansfield E.L., 
Symmetry reductions and exact solutions of a class of nonlinear heat equations,
{\it Physica D}, 1994, {\bf 70}, 250--288.

\bibitem{Clarkson&Mansfield1993a}
Clarkson P.A. and Mansfield E.L., 
Algorithms for the nonclassical method of symmetry reductions, 
{\it SIAM J. Appl. Math.}  1994, {\bf 54}, 1693--1719. 

\bibitem{Fushchych&Popowych1994-1}
Fushchych W.I. and Popowych R.O., 
Symmetry reduction and exact solution of the Navier--Stokes equations. I, 
{\it J. Nonlinear Math. Phys}, 1994, {\bf 1}, 75--113.

\bibitem{Fushchych&Serov1983}
Fushchych W.I., Serov N.I., 
The symmetry and some exact solutions of the nonlinear many-dimensional Liouville, d'Alembert and eikonal equations, 
{\it J. Phys. A: Math. Gen.}, 1983, {\bf 16}, 3645–-3658. 

\bibitem{Fushchych&Shtelen&Serov1993en}
Fushchych W.I., Shtelen W.M. and Serov N.I., 
{\it Symmetry analisys and exact solutions of equations of nonlinear mathimatical physics},  
Dordrecht, Kluwer Academic Publishers, 1993. 

\bibitem{Fushchych&Shtelen&Serov&Popovych1992}
Fushchych W. I., Shtelen W.M., Serov M.I. and Popovych R.O.,
$Q$-conditional symmetry of the linear heat equation,
{\it Proc. Acad. Sci. Ukraine}, 1992, no. 12, 28--33.

\bibitem{Fushchych&Tsyfra1987}
Fushchych W.I. and Tsyfra I.M.,
On a reduction and solutions of the nonlinear wave equations with broken symmetry,
{\it J. Phys. A: Math. Gen.}, 1987, {\bf 20}, L45--L48.

\bibitem{Fushchych&Zhdanov1992} 
Fushchych~W.I. and Zhdanov~R.Z., 
Conditional symmetry and reduction of partial differential equations, 
{\it Ukr. Math. J.}, 1992, {\bf 44}, 970--982.

\bibitem{Mansfield1999}
Mansfield E.L., The nonclassical group analysis of the heat equation,
{\it J. Math. Anal. Appl.}, 1999, {\bf 231}, 526--542.

\bibitem{Olver1993}
Olver P., {\it Applications of Lie groups to differential equations},
New-York, Springer-Verlag, 1993.

\bibitem{Olver1994}
Olver P., Direct reduction and differential constraints,
{\it Proc. R. Soc. Lond. A}, 1994, {\bf 444}, 509--523.

\bibitem{Olver&Vorob'ev1996}
Olver P.J. and Vorob'ev E.M., Nonclassical and conditional symmetries, 
in {\it CRC Handbook of Lie Group Analysis of Differential Equations}, 
Vol. 3, Editor N.H. Ibragimov, Boca Raton, Florida, CRC Press, 1996, 291--328.

\bibitem{Ovsiannikov1982}
Ovsiannikov~L.V., {\it Group analysis of differential equations},
New York, Academic Press, 1982.

\bibitem{Popovych1995}
Popovych R.O., 
On the symmetry and exact solutions of a transport equation, 
{\it Ukr. Math. J.}, 1995, {\bf 47}, 142--148.

\bibitem{Popovych1998}
Popovych R.O.,
On a class of $Q$-conditional symmetries and solutions of evolution equations,
in Symmetry and Analytic Methods in Mathematical Physics,
{\it Proceedings of Institute of Mathematics}, Kyiv, 1998, {\bf 19}, 194--199 (in Ukrainian).

\bibitem{Popovych2000}
Popovych R.O., 
Equivalence of $Q$-conditional symmetries under group of local transformation,
in Proceedings of the Third International Conference ``Symmetry in Nonlinear Mathematical Physics'' (Kyiv, July 12-18, 1999),
{\it Proceedings of Institute of Mathematics}, Kyiv, 2000, {\bf 30}, Part 1, 184--189; arXiv:math-ph/0208005.

\bibitem{Popovych2008a}
Popovych R.O., Reduction operators of linear second-order parabolic equations,  
{\it J. Phys. A}, 2008, {\bf 41}, 185202; arXiv:0712.2764.

\bibitem{Popovych&Vaneeva&Ivanova2007}
Popovych R.O., Vaneeva O.O and Ivanova N.M., 
Potential nonclassical symmetries and solutions of fast diffusion equation, 
{\it Phys. Lett. A}, 2007, {\bf 362}, 166--173; arXiv:math-ph/0506067.

\bibitem{Vasilenko&Popovych1999}
Vasilenko O.F. and Popovych R.O., 
On class of reducing operators and solutions of evolution equations, 
{\it Vestnik PGTU}, 1999, {\bf 8}, 269--273 (in Russian).

\bibitem{Webb1990}
Webb G.M., 
Lie symmetries of a coupled nonlinear Burgers-heat equation system,  
{\it J.~Phys. A: Math. Gen.}, 1990, {\bf 23}, 3885--3894.

\bibitem{Zhdanov&Lahno1998}
Zhdanov R.Z. and Lahno V.I.,
Conditional symmetry of a porous medium equation, {\it Phys. D}, 1998, {\bf 122}, 178--186.

\bibitem{Zhdanov&Tsyfra&Popovych1999}
Zhdanov R.Z., Tsyfra I.M. and Popovych R.O.,
A precise definition of reduction of partial differential equations,
{\it J. Math. Anal. Appl.}, 1999, {\bf 238}, 101--123; arXiv:math-ph/0207023.

\end{thebibliography}
\end{document}